\documentclass[12pt]{article}
\usepackage[utf8]{inputenc}
\usepackage{amsmath}
\usepackage{amsthm}
\usepackage{amssymb}
\usepackage[T1]{fontenc}

\usepackage{xspace}
\usepackage{graphicx}
\usepackage{yfonts}
\usepackage{amsthm}
\usepackage{amssymb}
\usepackage[all]{xy}
\usepackage{amsmath}
\usepackage{savesym}
\usepackage{amsmath}
\savesymbol{iint}
\usepackage{txfonts}
\usepackage{enumitem}
\newtheorem{theorem}{Theorem}[section]
\newtheorem{proposition}[theorem]{Proposition}
\newtheorem{lemma}[theorem]{Lemma}
\newtheorem{corollary}[theorem]{Corollary}
\newtheorem{remark}[theorem]{Remark}

\newcounter{cste}

\newcommand\g{\gamma\xspace}

\newcommand\gn{\gamma^{(n)}}
\newcommand\gku{\gamma^{(k+1)}}
\newcommand\gk{\gamma^{(k)}}
\newcommand\ddd{\partial}
\newcommand\G{\Gamma}

\newcounter{cster}

\newcommand\rr{\refstepcounter{cster}R_{\thecster}}

\newcounter{csterr}

\newcommand\rrr{\refstepcounter{csterr}R_{\thecsterr}}

\usepackage{url}

\newcounter{cstehigh}

\newcounter{csten}

\author{Mostapha Benhenda}
\title{Smooth Pseudo-Rotations Measure-Theoretically Isomorphic to Circle Rotations of Rationally Independent Angle}
\date{February 10, 2012}
\begin{document}

\maketitle

Let $M$ be a smooth compact connected manifold, on which there exists an effective smooth circle action $S_t$ preserving a positive smooth volume. In this chapter, we show that on $M$, the smooth closure of the smooth volume-preserving conjugation class of some Liouville rotations $S_\alpha$ of angle $\alpha$ contains a smooth volume-preserving diffeomorphism $T$ that is metrically isomorphic to an irrational rotation $R_\beta$ on the circle, with $\alpha \not\eq \pm \beta$, and with $\alpha$ and $\beta$ chosen either rationally dependent or rationally independent. In particular, if $M$ is the closed annulus $[0,1] \times \varmathbb{T}^1$, $M$ admits a smooth ergodic pseudo-rotation $T$ of angle $\alpha$ that is metrically isomorphic to the rotation $R_\beta$. Moreover, $T$ is smoothly tangent to $S_{\alpha} $ on the boundary of $M$.

\section{Introduction}

%Let $M$ be the annulus $\varmathbb{A}=[0,1] \times \varmathbb{T}^1$ or the torus $\varmathbb{T}^2$,

Let $\varmathbb{A}=[0,1] \times \varmathbb{T}^1$ be the closed annulus and $T$ be a homeomorphism isotopic to the identity. The \textit{rotation set} of $T$ measures the asymptotic speeds of rotation of the orbits of $T$ around the annulus. It generalizes the notion of rotation number of a circle homeomorphism, introduced by Poincar\'e. $T$ is an \textit{irrational pseudo-rotation} if its rotation set is reduced to a single irrational number $\alpha$, called the \textit{angle} of $T$. A broad question is raised by B\'eguin et al. \cite{beguinal04}: what are the similarities between the dynamics of the rigid rotation $S_\alpha$ of angle $\alpha$ and the dynamics of an irrational pseudo-rotation $T$ of angle $\alpha$? 

From a topological viewpoint, a similarity between $S_\alpha$ and $T$ (but with $T$ only $C^0$) has been shown by B\'eguin et al. \cite{beguinal04}: the rotation $S_\alpha$ is in the $C^0$-closure of the conjugacy class of $T$. Their result is analogous to a theorem by Kwapisz \cite{kwapisz03} on the torus $\varmathbb{T}^2$ (in this case, the angle of a pseudo-rotation is an element of $\varmathbb{T}^2$). J\"ager \cite{jager09} and Wang \cite{wang11} also investigated this broad question. However, there are also possible differences between $S_\alpha$ and $T$. From a metric viewpoint, Anosov and Katok \cite{anosovkatok70} constructed a smooth pseudo-rotation of $\varmathbb{A}$ that is metrically isomorphic to an ergodic translation of $\varmathbb{T}^2$. B\'eguin et al. \cite{beguinal07} constructed on $\varmathbb{T}^2$ a pseudo-rotation that is minimal, uniquely ergodic, but with positive entropy. In this paper, we construct a smooth pseudo-rotation of angle $\alpha$ that is metrically isomorphic to an irrational rotation $R_\beta$ with $\alpha \not\eq \pm \beta$. This is a construction of a non-standard smooth realization, based on the method of approximation by successive conjugations (see \cite{katokfayad04} for a presentation), a method that is often fruitful in smooth realization problems.

We recall that a \textit{smooth realization} of an abstract system $(X,f,\nu)$ is a triplet $(M,T,\mu)$, where $M$ is a smooth compact manifold, $\mu$ a smooth measure on $M$ and $T$ a smooth $\mu$-preserving diffeomorphism of $M$, such that $(M,T,\mu)$ is metrically isomorphic to $(X,f,\nu)$ (when $(M,\mu)$ and $(X,\nu)$ are implied, we just say that $T$ is metrically isomorphic to $f$). Moreover, a smooth realization is \textit{non-standard} if $M$ and $X$ are not diffeomorphic. 

Suppose there exists an ergodic pseudo-rotation $T$ of angle $\alpha$ that is a non-standard smooth realization  of a rotation $R_\beta$ on the circle. Then the couple $(\alpha,\beta)$ is called a \textit{non-standard couple of angles}. In this paper, we show that there exists non-standard couple of angles $(\alpha,\beta)$, such that $\alpha \not\eq \pm \beta$, with $\alpha$ and $\beta$ chosen either rationally dependent or rationally independent. 

Anosov and Katok \cite{anosovkatok70} showed the existence of an angle $\alpha$ such that $(\alpha,\alpha)$ is a non-standard couple of angles. Fayad et al. \cite{windsor07} showed that for any $\alpha$ Liouville, $(\alpha,\alpha)$ is a non-standard couple of angles. The question arises about the existence of a non-standard couple of angles $(\alpha, \beta)$ with $\alpha \not\eq \beta$.

It is worthy to recall that two ergodic rotations $R_\alpha$ and $R_\beta$ on the circle are metrically isomorphic if and only if $\beta= \pm \alpha$. If $\beta=\alpha$, the isomorphism is the identity, and if $\beta=-\alpha$, an isomorphism is given by a symmetry of axis going through the center of the circle. Therefore, by applying the result of Fayad et al. \cite{windsor07}, it becomes trivial to find a non-standard couple of angles $(\alpha,-\alpha)$. Our result shows that if, instead of considering metric automorphisms of the circle, we consider metric isomorphisms between the circle and the annulus, the situation becomes richer: we can have $\alpha \not\eq \pm \beta$, with $\alpha$ and $\beta$ either rationally dependent or rationally independent. However, $\alpha$ needs to be Liouville. Indeed, a result by Herman (with a proof published by Fayad and Krikorian \cite{fayadkrikorian09}) implies that if a smooth quasi-rotation $T$ of the closed annulus has Diophantine angle (i.e. non-Liouville), then $T$ cannot be ergodic (and a fortiori, $T$ cannot be metrically isomorphic to an ergodic rotation). However, the situation where $\alpha$ is Liouville and $\beta$ is Diophantine, though not addressed in this paper, is not excluded yet. The existence of this situation would reply positively to the open question about the existence of a non-standard smooth realization of a Diophantine circle rotation \cite{katokfayad04}.

More generally, let $M$ be a smooth compact connected manifold of dimension $d$, on which there exists an effective smooth circle action $S_t$ preserving a positive smooth measure $\mu$. Let $\mathcal{A}_\alpha$ be the smooth conjugation class of the rotation $S_\alpha$, and $\bar{\mathcal{A}_\alpha}$ its closure in the smooth topology. If $M=\varmathbb{T}^1$ and if $\alpha$ is Diophantine, then $\bar{\mathcal{A}_\alpha}=\mathcal{A}_\alpha$ by Herman-Yoccoz theorem \cite{yoccoz84} (indeed, by continuity, the rotation number of a diffeomorphism $T \in \bar{\mathcal{A}_\alpha}$ is $\alpha$). On the other hand, when $\alpha$ is Liouville, $\bar{\mathcal{A}_\alpha} \not\eq \mathcal{A}_\alpha$. In this paper, if $M$ has a dimension $d \geq 2$, then for some Liouville $\alpha$, we show that $\bar{\mathcal{A}_\alpha}$ contains non-standard smooth realizations of circle rotations $R_\beta$, with $\alpha \not\eq \pm \beta$, and with $\alpha$ and $\beta$ chosen either rationally dependent or rationally independent. In this case, $(\alpha,\beta)$ is still called a \textit{non-standard couple of angles}. More precisely, we show the following theorem:

%Let Diff$^\infty(M,\mu)Let $\mathcal{A}_\alpha= \{ B^{-1} S_\alpha B, B \in Diff$^\infty(M,\mu)  \}$ or of the disk $\varmathbb{D}^2$,

\begin{theorem}
\label{theoremdebaserotrotfonda}
Let $M$ be a smooth compact connected manifold of dimension $d \geq 2$, on which there exists an effective smooth circle action $(S_t)_{t \in \varmathbb{T}^1}$ preserving a positive smooth measure $\mu$. For any $u,v \in \varmathbb{T}^1$, for any $\epsilon >0$, there exist $(\alpha,\beta) \in \varmathbb{T}^1 \times \varmathbb{T}^1$ in a $\epsilon$-neighborhood of $(u,v)$, $T \in$ Diff$^\infty(M,\mu)$, such that $T \in \bar{\mathcal{A}_\alpha}$ and such that the rotation $R_\beta$ of angle $\beta$ on $\varmathbb{T}^1$ is metrically isomorphic to $T$. Moreover, $\beta$ can be chosen either rationally dependent or rationally independent of $\alpha$.
\end{theorem}

Theorem \ref{theoremdebaserotrotfonda} generalizes the particular case $M=[0,1]^{d-1} \times \varmathbb{T}^1$:

\begin{theorem}
\label{theoremdebaserotrotrotisom}
Let $d \geq 2$, $M= [0,1]^{d-1} \times \varmathbb{T}^1$, $\mu$ the Lebesgue measure. For $t \in \varmathbb{T}^1$, let $S_t: M \rightarrow M$ defined by $S_t(x,s)=(x,s+t)$. For any $u,v \in \varmathbb{T}^1$, for any $\epsilon >0$, there exist $(\alpha,\beta) \in \varmathbb{T}^1 \times \varmathbb{T}^1$ in a $\epsilon$-neighborhood of $(u,v)$, $T \in$ Diff$^\infty(M,\mu)$, such that for any $j \in \varmathbb{N}$, $(D^j T)_{|\ddd M}= (D^j S_{\alpha})_{ |\ddd M}$ and such that the rotation $R_\beta$ of angle $\beta$ on $\varmathbb{T}^1$ is metrically isomorphic to $T$. Moreover, $\beta$ can be chosen either rationally dependent or rationally independent of $\alpha$.
\end{theorem}

In the case of the closed annulus $M=[0,1] \times \varmathbb{T}^1$, we obtain:

\begin{corollary}
\label{corpseudorotrotisom}
Let $M= [0,1] \times \varmathbb{T}^1$, $\mu$ the Lebesgue measure. For $t \in \varmathbb{T}^1$, let $S_t: M \rightarrow M$ defined by $S_t(x,s)=(x,s+t)$. For any $u,v \in \varmathbb{T}^1$, for any $\epsilon >0$, there exist $(\alpha,\beta) \in \varmathbb{T}^1 \times \varmathbb{T}^1$ in a $\epsilon$-neighborhood of $(u,v)$, $T \in$ Diff$^\infty(M,\mu)$ a pseudo-rotation of angle $\alpha$, such that the rotation $R_\beta$ of angle $\beta$ on $\varmathbb{T}^1$ is metrically isomorphic to $T$. Moreover, $\beta$ can be chosen either rationally dependent or rationally independent of $\alpha$.
\end{corollary}

To show these results, we suitably modify one of Anosov and Katok's constructions. In \cite{anosovkatok70}, they constructed ergodic translations on the torus $\varmathbb{T}^h$, $h \geq 2$, of coordinates $(\beta_1,...,\beta_h)$, translations that admit non-standard smooth realizations on $[0,1]^{d-1} \times \varmathbb{T}^1$, $d \geq 2$, such that $T_{|\ddd M}$ is a rotation of angle $\alpha$. Moreover, in his construction, $\alpha \not\eq \beta_i$, $i=1,...,h$. In the previous chapter, we show that one $\beta_i$ can be an arbitrarily chosen Liouville number. However, this construction does not apply directly to the one-dimensional case. This is why, to obtain our result, though we essentially follow the previous chapter, we still need some substantial modifications.

\subsection{Definitions}

%An irrational number $\beta$ is \textit{Liouville} if, for any $k>0$, there is a sequence of integers $q_n \rightarrow +\infty$ such that $q_n^k \inf_{p \in \varmathbb{Z}}|q_n \beta - p| \rightarrow 0$. Liouville numbers are the complementary of Diophantine numbers in the set of irrational numbers. Let $\varmathbb{T}^h= \varmathbb{R}^h/\varmathbb{Z}^h$ denote the $h$-dimensional torus. Let $\mu_h$ be the Haar measure on $\varmathbb{T}^h$. 

%Let $d \geq 2$, $I= [0,1]^{d-1}$ and $M= I\times \varmathbb{T}^1$.
Let Diff$^\infty(M,\mu)$ be the class of smooth diffeomorphisms of $M$ preserving the Lebesgue measure $\mu$.  For $B \in$ Diff$^\infty(M,\mu)$ and $j \in \varmathbb{N}^*$, let $D^jB$ be the $j^{th}$ derivative of $B$ if $j >0$, and the $-j^{th}$ derivative of $B^{-1}$ if $j<0$. For $x \in M$, let $|D^jB(x)|$ be the norm of $D^jB(x)$ at $x$. We denote $\|B\|_k= \max_{0< |j| \leq k} \max_{ x \in M} |D^jB(x)|$.

A \textit{finite measurable partition} $\bar{\xi}$ of a measured manifold $(N, \nu)$ is the equivalence class of a finite set $\xi$ of disjoint measurable subsets of $N$ whose union is $N$, modulo sets of $\nu$-measure zero. In most of this paper, we do not distinguish a partition $\xi$ with its equivalent class $\bar{\xi}$ modulo sets of $\nu$-measure zero. In these cases, both are denoted $\xi$. Moreover, all partitions considered in this paper are representatives of a finite measurable partition. %The distance between two finite measurable partitions $\xi$ and $\xi'$ is defined by:

%\[ d(\xi,\xi')= \inf \sum_{c \in \xi, c' \in \xi'} \nu( c \Delta c') \] 

% (with respect to the measure $\mu$ on $M$).

%A \textit{measurable partition} of $M$ is a finite set $\xi$ of disjoints measurable (with respect to the measure $\mu$ on $M$) elements of $M$ whose union is $M$. All the partitions considered in this paper are measurable.% (with respect to the measure $\mu$ on $M$).

A partition $\xi'$ is \textit{subordinate} to a partition $\xi$ if any element of $\xi$ is a union of elements of $\xi'$, modulo sets of $\nu$-measure zero. In this case, if $\mathcal{B}(\xi)$ denotes the completed algebra generated by $\xi$, then $\mathcal{B}(\xi) \subset \mathcal{B}(\xi')$. The inclusion map $i: \mathcal{B}(\xi) \rightarrow \mathcal{B}(\xi')$ will be denoted $\xi \hookrightarrow \xi'$. This notation also means that $\xi'$ is \textit{subordinate} to $\xi$. A sequence of partitions $\xi_n$ is \textit{monotonic} if for any $n$, $ \xi_n \hookrightarrow \xi_{n+1}$. These definitions and properties are independent of the choice of the representatives $\xi$ and $\xi'$ of the equivalence classes $\bar{\xi}$ and $\bar{\xi'}$. 

A measure-preserving bijective bimeasurable map $T: (M_1,\mu_1, \mathcal{B}_1) \rightarrow (M_2,\mu_2, \mathcal{B}_2)$ induces an \textit{isomorphism} of measure algebras, still denoted $T: (\mu_1, \mathcal{B}_1) \rightarrow (\mu_2, \mathcal{B}_2)$. If $\xi_1,\xi_2$ are partitions, and if $\mathcal{B}_1=\mathcal{B}(\xi_1)$ and $\mathcal{B}_2=\mathcal{B}(\xi_2)$, we denote $T: \xi_1 \rightarrow \xi_2$ this induced isomorphism of measure algebras. If $M_1=M_2$, $\mu_1=\mu_2$ and $\mathcal{B}_1= \mathcal{B}_2$, then $T$ is a \textit{measure-preserving transformation}. Its induced isomorphism is an \textit{automorphism} (see \cite[p.43]{halmos56} and \cite{weiss72}).

A \textit{metric isomorphism} $L$ of measure-preserving transformations $T_1:(M_1,\mu_1, \mathcal{B}_1) \rightarrow (M_1,\mu_1, \mathcal{B}_1) $, $T_2:(M_2,\mu_2, \mathcal{B}_2) \rightarrow (M_2,\mu_2, \mathcal{B}_2)$ is a measure-preserving bijective bimeasurable map $L: (M_1,\mu_1, \mathcal{B}_1) \rightarrow (M_2,\mu_2, \mathcal{B}_2)$ such that $L T_1= T_2 L$ a.e. For convenience, when the measure is the Lebesgue measure and the algebra is the Borelian algebra, we omit to mention the measures and algebras, and we simply say that $L:(M_1,T_1) \rightarrow (M_2,T_2)$ is a metric isomorphism.

Let $\bar{\xi}$ be a measurable partition and $\xi$ a representative of this equivalent class modulo sets of $\mu$-measure zero. For $x \in M$, we denote $\xi(x)$ the element of the partition $\xi$ such that $x \in \xi(x)$. A sequences of partitions $\xi_n$ of measurable sets \textit{generates} if there is a set of full measure $F$ such that for any $x \in F$,

\[  \{ x \} = F \bigcap_{n \geq 1} \xi_n(x) \]

This property of generation is independent of the choice of the representatives $\xi_n$ of the equivalent class $\bar{\xi_n}$ and therefore, we will say that the sequence of measurable partitions $\bar{\xi_n}$ generates.
Let $M/\xi$ denote the equivalent class of the algebra generated by $\xi$, modulo sets of $\mu$-measure zero. $M/\xi$ is independent of the choice of the representative $\xi$ of the equivalent class $\bar{\xi}$. If $T: M_1 \rightarrow M_2$ is a measure-preserving map such that $T(\xi_1)=\xi_2$ $\mu$-almost everywhere, we can define a quotient map: $T/\xi_1: M/\xi_1 \rightarrow M/\xi_2$.

An \textit{effective action} of a group $G$ on $M$ is an action such that there is a set of full measure $F \subset M$ such that for any $x \in F$, there is $g \in G$ such that $gx \not\eq x$. A smooth effective circle action $(S_t)_{t \in \varmathbb{T}^1}$ on $M$ can be seen as a $1$-periodic smooth flow $(S_t)_{t \geq 0}$, we denote $\mathcal{A}_\alpha= \{ B^{-1} S_\alpha B, B \in \mbox{Diff}^\infty(M,\mu)  \}$. When $M=[0,1]^{d-1} \times \varmathbb{T}^1$, we consider the periodic flow $S_t$ defined by: 

\begin{eqnarray*}
S_t \colon [0,1]^{d-1} \times \varmathbb{T}^1  &\to [0,1]^{d-1} \times \varmathbb{T}^1  \\
(x,s) &\mapsto (x, t+s \mod 1)%, k=0,...,q_n-1
\end{eqnarray*}
%where $\bar{t}$ is the class of $t \in \varmathbb{R}$ modulo 1. 
For $a,b \in \varmathbb{T}^1$, let $[a,b[$ be the positively oriented circular sector between $a$ and $b$, with $a$ included and $b$ excluded.

%The \textit{diameter} $diam (\Gamma)$ of a domain $\Gamma \subset M$ is defined by: $diam(\Gamma)= \max_{x,y \in \Gamma} d(x,y)$, where $d(x,y)$ is the distance between $x$ and $y$.

%Let $h \geq 2$ and $\gamma=( \gamma_1,...,\gamma_h) \in \varmathbb{Z}^h$, with $\gamma_i, i=1,...,h$ relatively prime. Let $\{ T^{t \gamma} \}_{t \geq 0}$ a periodic flow on $\varmathbb{T}^h$. This flow has a fundamental domain $\Gamma \subset \varmathbb{T}^{h-1} \times \{ 0 \}$. The boundary of $\Gamma$ is of dimension $h-2$.

%For example, when $h=2$, there is a fundamental domain of the flow $\{ T^{t \gamma} \}_{t \geq 0}$ that is a segment line of length $1/\gamma_2$ (this can be seen using the Bezout identity for $(\gamma_1,\gamma_2)$). Remark that the $0$-volume of its boundary is equal to $2$ (the $0$-volume of a set of points is its cardinal).

A sequence $T_n$ of $\mu$-preserving maps \textit{weakly converges} to $T$ if, for any measurable set $E$, $\mu(T_n E \Delta E) \rightarrow 0$, where $A \Delta B= (A - B) \cup (B - A)$.

For $\g \in \varmathbb{R}$, we denote: $|\g|_{\mod 1} = \min_{k \in \varmathbb{Z}} | k+ \g |$

%For $a,b \in \varmathbb{T}^1$, we denote $[a,b[$ the positively oriented circular sector between $a$ and $b$, with $a$ included and $b$ excluded.

For $t \in \varmathbb{T}^1$ or $\varmathbb{R}$, and $A \subset I \times  \varmathbb{T}^1$, we denote \[ t+A = \left\{ ( x, t+s \mod 1), (x,s) \in A  \right\} \] 

Suppose $M=[0,1] \times \varmathbb{T}^1$ is the closed annulus. Let $\tilde{M}=[0,1] \times \varmathbb{R}$ be the universal covering of $M$. Let $T$ be a homeomorphism of $M$ isotopic to the identity and $\tilde{T}$ its lift to $\tilde{M}$. The \textit{rotation set} $\mbox{rot}( \tilde{T})$ of $\tilde{T}$ is defined by:
% and $p_2: [0,1] \times \varmathbb{R} \rightarrow \varmathbb{R}$ the second coordinate projection

\[  \mbox{rot}( \tilde{T}) = \bigcap_{k \geq 0} \bigcup_{n \geq k} \left\{ \frac{\tilde{T}^n(\tilde{x})- \tilde{x}}{n} / \tilde{x} \in \tilde{M} \right\} \]

We let the rotation set of $T$, $\mbox{rot}(T)$, be the equivalent class modulo 1 of $\mbox{rot}( \tilde{T})$. If $\mbox{rot}(T)$ is a singleton, and if $T$ is isotopic to the identity, then $T$ is a \textit{pseudo-rotation}. Note that, if $T_{| \ddd M }=S_{\alpha| \ddd M}$, then $T$ is isotopic to the identity. Indeed, 

$t \in [0,1] \mapsto S_{t\alpha}$ is a continuous path between the identity map and $S_\alpha$, and by Alexander's trick, any homeomorphism equal to the identity on the boundary is isotopic to the identity. In this paper, all the diffeomorphisms that we construct are equal to a rotation on the boundary and therefore, they are all isotopic to the identity.

\subsection{Basic steps of the proof}

The metric isomorphism of theorem \ref{theoremdebaserotrotrotisom} is obtained as the limit of isomorphisms of finite algebras: indeed, we use the lemma \cite[p.18]{anosovkatok70}:

\begin{lemma}
\label{lemmekatokrotisom}
Let $M_1$ and $M_2$ be Lebesgue spaces and let $\xi_n^{(i)}$ ($i=1,2$) be monotonic and generating sequences of finite measurable partitions of $M_i$. Let $T_n^{(i)}$ be automorphisms of $M_i$ such that $T_n^{(i)} \xi_n^{(i)}= \xi_n^{(i)}$ and $T_n^{(i)} \rightarrow T^{(i)}$ in the weak topology. Suppose there are metric isomorphisms $L_n: M_1 / \xi_n^{(1)} \rightarrow M_2 / \xi_n^{(2)} $ such that 

\[ L_n T_n^{(1)} / \xi_n^{(1)} =  T_n^{(2)} / \xi_n^{(2)} L_n \] 
and

\[ L_{n+1} \xi_n^{(1)} = \xi_n^{(2)} \]

then $(M_1,T_1)$ and $(M_2, T_2)$ are metrically isomorphic.

\bigskip

Said otherwise, if we have generating sequences of partitions and sequences of automorphisms $T_n^{(i)}$ weakly converging towards $T^{(i)}$, and if, for any integer $n$, the following diagram commutes:

\[ \xymatrix{
\xi_n^{(1)} \ar@(ul,dl)[]_{T_n^{(1)}}  \ar@{->}[r]^{L_n} \ar@{^{(}->}[d]  & \xi_n^{(2)} \ar@(ur,dr)[]^{T_n^{(2)}} \ar@{^{(}->}[d]  \\
\xi_{n+1}^{(1)}  \ar@{->}[r]^{L_{n+1}} & \xi_{n+1}^{(2)} 
 } \]

then $(M_1,T_1)$ and $(M_2, T_2)$ are metrically isomorphic.

\end{lemma}

%\ar@(ul,dl)[]_{T_{n+1}^{(1)}} \ar@(ur,dr)[]^{T_{n+1}^{(2)}} 

The proof of theorem \ref{theoremdebaserotrotrotisom} is in two steps. In the first step (lemma \ref{lemmefondarotisom}), we determine sufficient conditions on a sequence $(R_{\frac{p_n}{q_n} b_n} )_{n \geq 0}$ of periodic rotations of $\varmathbb{T}^1$ such that there exists sequences of finite partitions and automorphisms satisfying the assumptions of lemma \ref{lemmekatokrotisom} with $M_1= \varmathbb{T}^1$, $M_2=M$, $T_n^{(1)}=R_{\frac{p_n}{q_n} b_n}$, $T_n^{(2)}=T_n$, where $T_n$ is also smooth diffeomorphism, and such that the limit $T$ in the smooth topology of the sequence $T_n$ is smooth, and $T \in \bar{\mathcal{A}_\alpha}$ for $\alpha= \lim p_n/q_n$. 

In the second step (lemma \ref{lemme1point3rotisom}), we construct sequences of integers satisfying the conditions of the first step, such that $p_n/q_n \rightarrow \alpha$, $b_n p_n/q_n \rightarrow \beta$, with $(\alpha,\beta)$ that can be chosen arbitrarily close to any $(u,v) \in \varmathbb{T}^1 \times \varmathbb{T}^1$ , and with $(\alpha,\beta)$ either rationally dependent or rationally independent.

% show that the translation of vector $(\beta_1,...,\beta_{h-1},\beta)$ can be obtained as the limit of such sequences of translations $(T^{\frac{p_n}{q_n}\gn} )_{n \geq 0}$.

% M there is  be a Liouville number, $h\geq 2$ a positive integer. Let $M$ be a smooth compact connected manifold of dimension $d$, on which there exists an effective smooth circle action $S_t$ preserving a positive smooth measure $\mu$. There exists a dense set $E(\beta,d) \subset \varmathbb{T}^{h-1}$ such that for any $(\beta_1,..,\beta_{h-1}) \in E(\beta,d)$, there is   the ergodic translation of vector $(\beta_1,...,\beta_{h-1},\beta )$.

%\begin{remark}
%The proof of theorem \ref{theoremdebaserotrotrotisom} also gives that for any $n \in \varmathbb{N}^*$, any $u \in \varmathbb{T}$, any $\epsilon >0$, there is $\alpha \in \varmathbb{T}$ in a $\epsilon$-neighborhood of $u$, $T \in$ Diff$^\infty(M,\mu)$, such that $T_{|\ddd M}=S_{\alpha |\ddd M}$ and such that the rotation $R_{n\alpha}$ on $\varmathbb{T}$ is metrically isomorphic to $T$.
%\end{remark}

%The theorem follows from lemmas \ref{lemmefondarotisom} and \ref{lemme1point3rotisom}:

%The proof of the theorem is in two steps. satisfy the assumptions of lemma \ref{lemmefondarotisom}. We show that can be obtained as the limit of the sequences of satisfying  $\rr(n) \label{r1cste}$,

\begin{lemma} 
\label{lemmefondarotisom}
There exists an explicit sequence of integers $\rr(n)\geq n \label{r0csterotisom}$, such that, if there exist increasing sequences of integers $p_n,q_n, a_n,b_n \in \varmathbb{N}^*$, and a sequence $s_n  \in \varmathbb{Z}^*$ such that, for any integer $n$,

\begin{enumerate}
\item \textit{(primality)} \label{61rotisom} $a_nb_n-s_nq_n=1$.
\item \textit{(monotonicity)} \label{63rotisom}$q_n$ divides $q_{n+1}$ and $q_n<q_{n+1}$.
%\item \textit{(horizontal monotonicity)} \label{64} $\gamma_h^{(n)}$ divides $\gamma_h^{(n+1)}$
\item \textit{(isomorphism)} \label{65rotisom} $q_n$ divides $a_{n+1}-a_n$.
\item \textit{(convergence of the diffeomorphism, generation)} \label{66rotisom} \[ 0< \left| \frac{p_{n+1}}{q_{n+1}} - \frac{p_n}{q_n} \right| \leq \frac{1}{(b_{n+1} q_n)^{R_{\ref{r0csterotisom}}(n)}}  \]
%\item \textit{(horizontal convergence of the partition)} \label{67} Let $\G^{(n)} \subset \varmathbb{T}^{h-1} \times \{ 0 \}$ be a fundamental domain of the flow $\{T^{t\gn} \}$,  $d_{n}$ be the diameter of $\G^{(n)}$, and $\sigma_{n}$ the $(h-2)$-dimensional volume of the boundary of $\G^{(n)}$. Then

%\[ d_{n+1} \leq \frac{1}{2^n \gamma_h^{(n)}  \sigma_n  }  \]

%\item \textit{(vertical convergence of the partition)} \label{68} \[  \sum_{n \geq 0}  \frac{ (\gamma_h^{(n)})^2  \sigma_n  }{q_n}   \left| \frac{\gnu}{ \gamma_h^{(n+1)} } - \frac{\gn}{\gamma_h^{(n)}} \right| < + \infty    \]
\end{enumerate}

then all these assumptions imply that there are $\alpha,\beta \in \varmathbb{T}^1$ such that \[ \frac{p_n}{q_n} \rightarrow^{mod 1} \alpha, \;  \frac{p_n}{q_n} b_n \rightarrow^{mod 1} \beta \]

and there is a smooth ergodic measure-preserving diffeomorphism $T$ of $M$ such that for any $j \in \varmathbb{N}$, $(D^jT)_{|\ddd M}=(D^jS_{\alpha})_{ |\ddd M}$ and such that $(\varmathbb{T}^1, R_\beta, Leb)$ is metrically isomorphic to $(M,T,\mu)$.

\end{lemma}

\begin{lemma}
\label{lemme1point3rotisom}
For any $u,v \in \varmathbb{T}^1$, for any $\epsilon >0$, there exist $(\alpha,\beta) \in \varmathbb{T}^1 \times \varmathbb{T}^1$ in a $\epsilon$-neighborhood of $(u,v)$, such that there exist sequences of integers $p_n,q_n, a_n,b_n \in \varmathbb{N}^*$, $s_n  \in \varmathbb{Z}^*$ satisfying the assumptions of lemma \ref{lemmefondarotisom}, such that \[ \frac{p_n}{q_n} \rightarrow^{mod 1} \alpha, \;  \frac{p_n}{q_n} b_n \rightarrow^{mod 1} \beta \]

Moreover, $\beta$ can be chosen either rationally dependent of $\alpha$ or rationally independent of $\alpha$.

\end{lemma}

%In the first part of the proof, we construct a monotonic and generating partition on the $h$-dimensional torus $\varmathbb{T^h}$, called $\zeta_n^\infty$, which is adapted to the translation $T^{\frac{p_n}{q_n}\gamma^{(n)}}$. To that end, we use assumptions $\ref{63rotisom}$, $\ref{64}$, $\ref{67}$ and $\ref{68}$.
%The metric isomorphism of lemma \ref{lemmefondarotisom} is obtained as the limit of finite isomorphisms of finite algebras: indeed, we have the lemma \cite[p.18]{anosovkatok70}: 

%Lemma \ref{lemmekatokrot} involves monotonic and generating partitions $\xi_n^1$ and $\xi_n^2$ on each space, automorphisms $T_n^{(1)}$, $T_n^{(2)}$ stabilizing each partition, and isomorphisms $L_n$ that send each partition $\xi_n^1$ on its counterpart $\xi_n^2$. 

We divide the proof of lemma \ref{lemmefondarotisom} in two main parts. In the first part of the proof, we elaborate sufficient conditions on $B_n \in$ Diff$^\infty(M,\mu)$, where $M=[0,1]^{d-1} \times \varmathbb{T}^1$, so that if $T_n= B_n^{-1} S_{\frac{p_n}{q_n}} B_n$ weakly converges towards an automorphism $T$, then there exists a metric isomorphism between $(\varmathbb{T}^1,R_\beta,Leb)$ and $(M,T,\mu)$. To that end, we apply lemma \ref{lemmekatokrotisom}: we construct a monotonous and generating sequence of partitions $\xi_n^\infty$ of $M$ and a sequence of isomorphisms $\bar{K}_n^\infty: \varmathbb{T}^1/\zeta_n \rightarrow M/ \xi_n^\infty$, where $\zeta_n=\{ [i/q_n, (i+1)/q_n[,i=0,...,q_n-1  \}$, such that $ \bar{K}_n^\infty R_{\frac{p_n}{q_n}} = T_n\bar{K}_n^\infty $ and $\bar{K}_{n+1|\zeta_n}^\infty=\bar{K}_n^\infty$. In the construction of this isomorphism, assumption $\ref{65rotisom}$ is important. Moreover, we will see that the elements of $\xi_n^\infty$ are not the most elementary, because they must be chosen in a way that ensures the monotonicity of the sequence $\bar{K}_n^\infty$. This condition of monotonicity induces combinatorial constraints on the elements of the partition $\xi_n^\infty$. Though it follows a similar scheme, the construction of the sequence $\bar{K}_n^\infty$ differs from the construction done in the previous chapter, especially because the assumption \ref{61rotisom} is new.

In the second part of the proof, we construct diffeomorphisms $T_n= B_n^{-1} S_{\frac{p_n}{q_n}}B_n $ on $M$ stabilizing $\xi_n^\infty$, obtained by successive conjugations from the rotation $S_{\frac{p_n}{q_n}}$. The conjugacy $B_n$ is constructed explicitly. In this second part, we essentially follow the previous chapter (which elaborated on \cite{windsor07}), except for the obtention of the generation of the sequence of partitions $(\xi_n^\infty)_{n \geq 1}$, for which we need to slightly modify the construction. 

Another change with respect to the previous chapter is in the construction of the limit angles $\alpha$ and $\beta$, 
i.e. in the proof of lemma \ref{lemme1point3rotisom}.

\subsection{Construction of the limit angles $\alpha$ and $\beta$: proof of lemma \ref{lemme1point3rotisom}}

\subsubsection{The case $\beta=p \alpha$}
We show:
\begin{lemma}

For any $u,v \in \varmathbb{T}^1$, for any $\epsilon >0$, there exist $(\alpha,\beta) \in \varmathbb{T}^1 \times \varmathbb{T}^1$ in a $\epsilon$-neighborhood of $(u,v)$, such that there exist sequences of integers $p_n,q_n, a_n,b_n \in \varmathbb{N}^*$, $s_n  \in \varmathbb{Z}^*$ satisfying the assumptions of lemma \ref{lemmefondarotisom}, such that \[ \frac{p_n}{q_n} \rightarrow^{mod 1} \alpha, \;  \frac{p_n}{q_n} b_n \rightarrow^{mod 1} \beta \]

Moreover, $\beta$ can be chosen rationally dependent of $\alpha$.

\end{lemma}

\begin{proof}

Let $u,v \in \varmathbb{T}^1$ and $\epsilon >0$. Let $p_0,q_0,b_0$ be positive integers such that $\gcd(b_0,q_0)=1$, and such that: 

\[ \left|  \frac{p_0}{q_0} -u \right|_{\mod 1} \leq \frac{\epsilon}{2}, \; \left|  \frac{p_0b_0}{q_0} -v \right|_{\mod 1} \leq \frac{\epsilon}{2}  \] 

By the Bezout theorem, there are integers $a_0,s_0$, with $a_0 > 0$, such that $a_0 b_0 - s_0 q_0=1$.

Suppose we have defined $p_k,q_k,a_k,b_k,s_k$, satisfying the assumptions of lemma \ref{lemmefondarotisom}, up to the rank $k=n$, and let us define $p_{n+1},q_{n+1},a_{n+1},b_{n+1},s_{n+1}$. (we will have $s_n=1$ for $n \geq 1$). Let $b_{n+1}=b_n$.

Let $c_n$ be an integer sufficiently large so that $c_n \geq (b_{n+1} q_n)^{R_{\ref{r0csterotisom}}(n)}$ and $c_n \geq b_{n+1} 2^{n+1}/\epsilon$ ($b_n=b_0$  is constant here, but this more general definition is used for the case $(\alpha,\beta)$ rationally independent). Let

\[ a_{n+1}= a_n + s_n c_n q_n \]

Therefore, assumption \ref{65rotisom} holds. Let also 

\[ q_{n+1}= q_n s_n (1+c_n b_n) \]

Therefore, assumption \ref{63rotisom} holds. Moreover, we have:

\[ a_{n+1}b_{n+1}-q_{n+1}=1   \]

Therefore, assumption \ref{61rotisom} holds, with $s_{n+1}=1$. Moreover, let $p_{n+1}= p_n \frac{q_{n+1}}{q_n} +1$. Since $q_{n+1} \geq (b_{n+1} q_n)^{R_{\ref{r0csterotisom}}(n)}$, we have:

\[ 0< \left| \frac{p_{n+1}}{q_{n+1}} - \frac{p_n}{q_n} \right|= \frac{1}{q_{n+1}} \leq \frac{1}{(b_{n+1} q_n)^{R_{\ref{r0csterotisom}}(n)}}  \]

Therefore, assumption \ref{66rotisom} holds. Moreover,

\[ \frac{p_n}{q_n} b_n = \frac{p_0}{q_0}b_0 + \sum_{k=0}^{n-1} \left( \frac{p_{k+1}}{q_{k+1}} b_{k+1} - \frac{p_k}{q_k} b_k \right) =_{\mod 1} \frac{p_0}{q_0}b_0 +   \sum_{k=0}^{n-1} \left( \frac{p_{k+1}}{q_{k+1}} - \frac{p_k}{q_k} \right) b_{k+1}=   \frac{p_0}{q_0}b_0 +   \sum_{k=0}^{n-1} \frac{b_{k+1}}{q_{k+1}} \]

Since $1/q_{n+1} \leq \epsilon/ (2^{n+1} b_{n+1})$, we get:
 \[ \left| \frac{p_{n}}{q_{n}}b_n - \frac{p_0}{q_0}b_0 \right| \leq \frac{\epsilon}{2} \]

Therefore, \[\frac{p_{n}}{q_{n}}b_n \rightarrow \beta \]

with $|\beta-u| \leq \epsilon$

Likewise, 

\[\frac{p_{n}}{q_{n}} \rightarrow \alpha \]

with $|\alpha-v| \leq \epsilon$. Moreover, we have $\beta= b_0 \alpha$.

\end{proof}

\subsubsection{The case $(\alpha,\beta)$ rationally independent}

We show:

\begin{lemma}

For any $u,v \in \varmathbb{T}^1$, for any $\epsilon >0$, there exist $(\alpha,\beta) \in \varmathbb{T}^1 \times \varmathbb{T}^1$ in a $\epsilon$-neighborhood of $(u,v)$, such that there exist sequences of integers $p_n,q_n, a_n,b_n \in \varmathbb{N}^*$, $s_n  \in \varmathbb{Z}^*$ satisfying the assumptions of lemma \ref{lemmefondarotisom}, such that \[ \frac{p_n}{q_n} \rightarrow^{mod 1} \alpha, \;  \frac{p_n}{q_n} b_n \rightarrow^{mod 1} \beta \]

Moreover, $\beta$ can be chosen rationally independent of $\alpha$.

\end{lemma}

\begin{proof}
The beginning of the construction is the same as in the case $\beta=p\alpha$, except that we take:

\[ b_{n+1}=b_n+q_n \]

\[ q_{n+1}= s_n q_n ( 1+c_nb_n +c_n q_n + a_n ) \]

This ensures that $b_n \rightarrow + \infty$ as $n \rightarrow + \infty$, and that 
\[ a_{n+1}b_{n+1}-q_{n+1}=1   \]

It only remains to show that the limit angles $(\alpha,\beta)$ are rationally independent. To that aim, it suffices to show that the translation of vector $(\alpha,\beta)$ on the torus $\varmathbb{T}^2$ is ergodic. We follow the proof of the ergodicity of the limit translation in the previous chapter, with a slight modification. We recall a theorem by Katok and Stepin \cite{katokstepin67}:

\begin{theorem}[Katok-Stepin \cite{katokstepin67}]
\label{katokstepintheoremrotisom}
Let $U$ be an automorphism of a Lebesgue space $(N,\nu)$, let $(U_n)_{ n \geq 1}$ be a sequence of measure-preserving transformations, and let $(\chi_n)_{ n \geq 1}$ be a sequence of finite partitions of $N$ with measurable elements. Suppose that:

\begin{itemize}
\item $U_n$ permutes the elements of $\chi_n$ cyclically.
\item $(\chi_n)_{ n \geq 1}$ generates.
\item $\sum_{c \in \chi_n} \nu \left( U(c) \Delta U_n (c)  \right) = o(1/|\chi_n|)$ (where $|\chi_n|$ is the cardinal of $\chi_n$).
\end{itemize}
then $U$ is ergodic.
\end{theorem}

Let $\gn= (1,b_n)$, $g_n = \gcd (p_n,q_n)$. Let $\Gamma^{(n)} \subset \varmathbb{T}^{2}$ a fundamental domain of the flow $(T^{t\gamma^{(n)}})_{t \geq 0}$ on $\varmathbb{T}^2$, where $T^{t\gamma^{(n)}}$ is the translation of vector $t\gamma^{(n)}$. Note that the diameter of $\Gamma^{(n)} $ is less than $1/b_n$. Let 
 
\[ \Gamma_{0,n} = \bigcup_{0 \leq t < \frac{g_n}{q_n}} T^{t\gamma^{(n)}} \Gamma^{(n)}  \]

%\[  \hat{\zeta}_n = \{  \Gamma_{k,n}= T^{\frac{g_n}{q_n}\gamma^{(n)}} \Gamma_{0,n}, 0 \leq k \leq \frac{q_n}{g_n} -1  \} \] 

We have the lemma:

\begin{lemma}
\label{lemmehatzetanrotisom}

Let $\hat{\zeta}_n$ be the partition defined by:

\[  \hat{\zeta}_n = \left\{ \Gamma_{i,n}=  T^{i \frac{g_n \gn}{q_n}} \Gamma_{0,n}, \, i=0,...,\frac{q_n}{g_n}-1  \right\} \]
$T^{\frac{p_n}{q_n} \gn}$ is a cyclic permutation on $\hat{\zeta}_n$, and $\hat{\zeta}_n$ generates.
\end{lemma}

\begin{proof}
$T^{\frac{p_n}{q_n} \gn}$ is a cyclic permutation on $\hat{\zeta}_n$ because $g_n = \gcd (p_n,q_n)$.

To the vector space $\varmathbb{R}^2$, we give the norm $\|(x_1,x_2)\|= \max_{1 \leq i \leq 2} |x_i| $ and we consider its induced norm on $\varmathbb{T}^2$.

Since

\[ p_{n+1} - \frac{q_{n+1}}{q_n} p_n = 1 \]

then $ p_{n+1}$ and $\frac{q_{n+1}}{q_n}$ are relatively prime. Since $g_{n+1}$ divides $p_{n+1}$ and $q_{n+1}$, then $g_{n+1}$ divides $q_{n}$. In particular, $g_{n+1} \leq q_n$ (this is the slight difference with the proof in the previous chapter, in which we do not have: $\gcd( p_{n+1},\frac{q_{n+1}}{q_n})=1$. But on the other hand, in the previous chapter, we have: $\gcd( p_{n+1},\frac{q_{n+1}}{q_n})=\gcd( p_{n+1},q_{n+1})$. The important point is that in both cases, $g_n$ is small enough).

Moreover, by assumption \ref{66rotisom} of lemma \ref{lemmefondarotisom}, \[ \frac{b_{n+1}q_n}{q_{n+1}} \leq \frac{1}{(b_{n+1} q_n)^{R_{\ref{r0csterotisom}}(n)-1}} \rightarrow_{n \rightarrow + \infty} 0 \]
Therefore,

\[ diam(\Gamma_{0,n}) \leq \max\left( \frac{1}{b_n}, \frac{g_n \| \gn \| }{q_n} \right) \leq \max\left( \frac{1}{b_n}, \frac{q_{n-1} b_n}{q_n} \right) \rightarrow_{n \rightarrow + \infty} 0 \]

It shows that $\hat{\zeta}_n$ generates.
\end{proof}

It remains to estimate $\sum_{c \in \hat{\zeta}_n} \mu_2 \left( T^{ \alpha } c \Delta  T^{ \frac{p_{n}}{q_{n}} \gn } c  \right)$, where $\mu_2$ is the Lebesgue measure on $\varmathbb{T}^2$. We have the lemma:

\begin{lemma}
\label{speedofapproximrotisom}
%There is a choice of $R_{\ref{r0csterotisom}}(n)$ in condition \ref{65rotisom} of lemma \ref{lemmefondarotisom} such that:

We have:
\[ \sum_{c \in \hat{\zeta}_n} \mu_2 \left( T^{ (\alpha,\beta) } c \Delta  T^{ \frac{p_{n}}{q_{n}} \gn } c  \right) =  o(g_n/q_n)= o(1/|\hat{\zeta}_n|) \]

\end{lemma}

\begin{proof}

We have:

\[ \sum_{c \in \hat{\zeta}_n} \mu_2 \left( T^{ (\alpha,\beta) } c \Delta  T^{ \frac{p_{n}}{q_{n}} \gn } c  \right)= \sum_{k \geq n} \sum_{c \in \hat{\zeta}_k} \mu_2 \left( T^{ \frac{p_{k+1}}{q_{k+1}} \gku } c \Delta  T^{ \frac{p_{k}}{q_{k}} \gk } c  \right)= \sum_{k \geq n} \sum_{c \in \hat{\zeta}_k} \mu_2 \left( T^{ \frac{p_{k+1}}{q_{k+1}} \gku - \frac{p_{k}}{q_{k}} \gk } c \Delta c  \right) \]

Let $\tau_k$ be the $(h-1)$-volume of the border of an element of $\hat{\zeta}_k$. We have:

\[ \mu_2 \left( T^{ \frac{p_{k+1}}{q_{k+1}} \gku - \frac{p_{k}}{q_{k}} \gk } c \Delta c  \right) \leq \tau_k \left\| \frac{p_{k+1}}{q_{k+1}} \gku - \frac{p_{k}}{q_{k}} \gk \right\| = \tau_k \left\|  \left( \frac{p_{k+1}}{q_{k+1}}  - \frac{p_{k}}{q_{k}} \right)  \gku \right\|= \tau_k \frac{b_{k+1}}{q_{k+1}} \]

Moreover, \[ \tau_k \leq 2 \left( \frac{1}{b_k} + \frac{g_kb_k}{q_k}  \right) \leq 2 \left( \frac{1}{b_k} + \frac{q_{k-1}b_k}{q_k}  \right) \leq 4 \]

Therefore, by applying assumption \ref{66rotisom},  %there is a choice of $R_{\ref{r0csterotisom}}(n)$ such that:

\[ \sum_{c \in \hat{\zeta}_n} \mu_2 \left( T^{ (\alpha,\beta) } c \Delta  T^{ \frac{p_{n}}{q_{n}} \gn } c  \right) =  o(g_n/q_n) \]

\end{proof} 

%there is a choice of $R_{\ref{r0csterotisom}}(n)$ such that 
By combining lemmas \ref{lemmehatzetanrotisom} and \ref{speedofapproximrotisom}, and by applying theorem \ref{katokstepintheoremrotisom}, we obtain that the translation of vector $(\alpha,\beta)$ is ergodic with respect to the Lebesgue measure.

This completes the proof of lemma \ref{lemme1point3rotisom}.
\end{proof}

\bigskip

\subsubsection{Why not a construction for all $\alpha$ Liouville?}

Let us make one remark. We were not able to show our theorem for any $\alpha$ Liouville, because conditions \ref{61rotisom}-\ref{65rotisom} of lemma \ref{lemmefondarotisom} introduce arithmetical constraints on the denominators of the convergents of $\alpha$. These conditions are analogous to those, in the previous chapter, which limit the set of possible translations of the $h$-dimensional torus, $h\geq 2$, that admit a non-standard smooth realization. 

A sufficient condition for $\alpha$ Liouville to belong to a non-standard couple of angles $(\alpha,\beta)$ with $\alpha \not\eq \pm \beta$, is the following: if, for the sequence $p'_n,q'_n$ of convergents of $\alpha$, there exist positive integers $c_n,d_n$, with $d_n \leq q_n^{R(n)}$ for a fixed sequence $R(n)$, such that we can write: 

\[ q'_{n+1}= 1+c_nb_n +d_n a_n+ c_n d_n q_n \]

then there is $\beta \not\eq \pm \alpha$ such that $(\alpha,\beta)$ is a non-standard couple of angles. (in this construction, we take $b_{n+1}=b_n + d_n q_n$, with $d_n \in \varmathbb{N}$)

\bigskip

\subsection{Convergence modulo 1 of $\frac{p_n}{q_n}$ and $  \frac{p_n}{q_n} b_n$ towards $\alpha$ and $\beta$.}

The rest of the paper is dedicated to the proof of lemma \ref{lemmefondarotisom}. Part of lemma \ref{lemmefondarotisom} is straightforward, namely, the convergence modulo 1 of $\frac{p_n}{q_n}$ and $  \frac{p_n}{q_n} b_n$ towards $\alpha$ and $\beta$ respectively:

\begin{proof}[Partial proof of lemma \ref{lemmefondarotisom}.]
By assumption $\ref{63rotisom}$, for $n \geq 2$, $q_n \geq 2$. By assumption $\ref{65rotisom}$, and since $R_{\ref{r0csterotisom}}(n) \geq n$, $p_n/q_n$ is Cauchy, and converges.

To show the convergence modulo 1 of $\frac{p_n}{q_n} b_n$, we note that assumptions $\ref{61rotisom}$ at ranks $n$ and $n+1$, and assumption $\ref{65rotisom}$ at rank $n$ imply that $q_n$ divides $b_{n+1}-b_n$. Indeed, let us write $b_{n+1}=b_n+k$, with $k$ integer, and let us show that $q_n$ divides $k$. By the assumption $\ref{65rotisom}$ at rank $n$, $a_{n+1}=a_n+c_n q_n$, with $c_n$ integer. Therefore,

\[ 1+s_{n+1} q_{n+1}=a_{n+1}b_{n+1}= (a_n+c_n q_n)(b_n+k)=a_nb_n+ a_nk + q_n(c_n b_n+c_nk)   \]

Therefore,
 
 \[  q_n \left(s_{n+1} \frac{q_{n+1}}{q_n}-s_n-c_n b_n -c_nk \right)=a_n k \]

Thus, $q_n$ divides $a_nk$. Since $q_n$ is relatively prime with $a_n$, then $q_n$ divides $k$. Therefore, $b_n/q_n=b_{n+1}/q_n$ mod 1. Therefore,

\[  \left| \frac{p_{n+1}b_{n+1}}{q_{n+1}} - \frac{p_nb_n}{q_n} \right|=^{mod 1}  \left| \frac{p_{n+1}}{q_{n+1}} - \frac{p_n}{q_n} \right| |b_{n+1}| \leq  \frac{1}{(b_{n+1} q_n)^{R_{\ref{r0csterotisom}}(n)-1}}   \]

Since for $n \geq 1$, $q_n \geq 2$ and $R_{\ref{r0csterotisom}}(n)-1 \geq n-1$, then the sequence $\left(\frac{p_nb_n}{q_n} \, \mod 1 \, \right)_{n \geq 1}$ is Cauchy, and converges.

\end{proof}

To show lemma \ref{lemmefondarotisom}, it remains to show that there is a smooth ergodic measure-preserving diffeomorphism $T$ of $M$ such that $T \in \bar{\mathcal{A}_{\alpha}}$  and such that $(\varmathbb{T}^1, R_\beta, Leb)$ is metrically isomorphic to $(M,T,\mu)$. 

% between the rotation $R_\beta$ and the ergodic diffeomorphism with $T_{|\ddd M}=S_{\alpha |\ddd M}$$T_{|\ddd M}=S_{\alpha |\ddd M}$ with $(D^jT)_{|\ddd M}=(D^jS_{\alpha})_{|\ddd M}$ for any $j \in \varmathbb{N}$,

\section{The metric isomorphism}

In this section, our aim is to elaborate sufficient conditions on $B_n \in$ Diff$^\infty(M,\mu)$, where $M=[0,1]^{d-1} \times \varmathbb{T}^1$, so that if $T_n= B_n^{-1} S_{\frac{p_n}{q_n}} B_n$ weakly converges towards an automorphism $T$, then there exists a metric isomorphism between $(\varmathbb{T}^1,R_\beta,Leb)$ and $(M,T,\mu)$.

%In this section, we construct the isomorphism between the rotation $R_\beta$ on $\varmathbb{T}^1$ and the ergodic diffeomorphism   on $M=[0,1]^{d-1} \times \varmathbb{T}^1$. 

To that end, we use lemma \ref{lemmekatokrotisom}: we construct a monotonous and generating sequence of partitions $\xi_n^\infty$ of $M$ and a sequence of isomorphisms $\bar{K}_n^\infty: \varmathbb{T}^1/\zeta_n \rightarrow M/ \xi_n^\infty$, where $\zeta_n=\{ [i/q_n, (i+1)/q_n[,i=0,...,q_n-1  \}$, such that $ \bar{K}_n^\infty R_{\frac{p_n}{q_n}} = T_n\bar{K}_n^\infty $ and $\bar{K}_{n+1|\zeta_n}^\infty=\bar{K}_n^\infty$. 

$\zeta_n$ is a partition of $\varmathbb{T}^1$ that is monotonic (because $q_n$ divides $q_{n+1}$) and that generates (because $q_n \rightarrow + \infty$). Let $\eta_n= \{ I \times [j/q_n, (j+1)/q_n[ , j=0,...,q_n-1 \}$. $\eta_n$ is a monotonic partition of $M$. %$\eta_n$ is monotonic and generates because  (because $\varmathbb{T}^1$Since $a_n$ and $q_n$ are relatively prime,

The following lemma is straightforward, but important: 

\begin{lemma}
\label{lempermutation}
Let $a_n$ and $q_n$ two relatively prime integers, and let

\[ K_n : \begin{array}[t]{lcl} \zeta_n  &\rightarrow &   \eta_n  \\
               \left[ \frac{i}{q_n}, \frac{i+1}{q_n} \right[  & \mapsto    & I \times \left[ \frac{ia_n}{q_n}, \frac{ia_n+1}{q_n} \right[ 
           \end{array}
           \]

$K_n$ is a metric isomorphism such that $K_nR_{\frac{1}{q_n}} = S_{\frac{a_n}{q_n}} K_n$. In other words, the following diagram commutes:

\[ \xymatrix{
 \zeta_n \ar@(ul,dl)[]_{R_{\frac{1}{q_n}}}  \ar@{->}[r]^{K_n} & \eta_n \ar@(ur,dr)[]^{S_{\frac{a_n}{q_n}}} 
 } \]

\end{lemma}

This lemma is related with two basic observations: the first is that both $R_{\frac{1}{q_n}}$ and $S_{\frac{a_n}{q_n}}$ are isomorphic to cyclic permutations of $\{ 0,...,q_n-1\}$ (this set is given the counting measure, i.e. $\mu (A)=\# A$); the second observation is that two cyclic permutations of the same order are always conjugated.

\bigskip

The following lemma combines lemma \ref{lempermutation} with the facts that $\zeta_n  \hookrightarrow  \zeta_{n+1}$ and 

$\eta_n \hookrightarrow  \eta_{n+1}$:

\begin{lemma}
\label{diagcomutrotisom}
Let $a_n,a_{n+1},q_n,q_{n+1} \in \varmathbb{N}$ such that $\gcd(a_n,q_n)=\gcd(a_{n+1},q_{n+1})=1$, such that $q_n$ divides $q_{n+1}$ and such that $q_n$ divides $a_{n+1}-a_n$. There exists a partition $\eta_n^{n+1} \hookrightarrow  \eta_{n+1}$ of $M$ stable by $S_{\frac{a_n}{q_n}}$, and there exists a metric isomorphism $K_n^{n+1}: \zeta_n \rightarrow \eta_n^{n+1}$ such that $K_n^{n+1}=K_{n+1|\zeta_n}$ and such that $K_n^{n+1} R_{\frac{1}{q_n}} =S_{\frac{a_n}{q_n}} K_n^{n+1}$. There exists also a metric isomorphism $C_n^{n+1}: \eta_n \rightarrow \eta_n^{n+1}$ such that $C_n^{n+1} S_{\frac{a_n}{q_n}} =S_{\frac{a_n}{q_n}} C_n^{n+1}$ and $K_n^{n+1}=C_n^{n+1}K_n$. Said otherwise, we have the following commutative diagram:

\[ \xymatrix{
  \zeta_n \ar@(ul,dl)[]_{R_{\frac{1}{q_n}}}  \ar@{->}[r]^{K_n} \ar@{->}[d]_{Id} & \eta_n \ar@(ur,dr)[]^{S_{\frac{a_n}{q_n}}} \ar@{->}[d]^{C_n^{n+1}} \\
 \zeta_n \ar@(ul,dl)[]_{R_{\frac{1}{q_n}}}  \ar@{->}[r]^{K_n^{n+1}} \ar@{^{(}->}[d]  & \eta_n^{n+1} \ar@(ur,dr)[]^{S_{\frac{a_n}{q_n}}} \ar@{^{(}->}[d]  \\
 \zeta_{n+1}  \ar@{->}[r]^{K_{n+1}}  & \eta_{n+1} 
 } \]

%\xymatrix{
% \zeta_n \ar@{->}[r]^{K_n} \ar@{^{->}[d]_{Q_n^{n+1}} & \eta_n \ar@{->}[d]^{C_n^{n+1}} \\
% \zeta_n^{n+1}  \ar@{^{(}->}[d]  \ar@{->}[r]^{K_n^{n+1}}  & \eta_n^{n+1} \ar@(dl,ul)[]|{S_{\frac{1}{q_n}}}
% }

%\ar@(ur,dr)[]^{S_{\frac{a_{n+1}}{q_{n+1}}}}\ar@(ul,dl)[]_{R_{\frac{1}{q_{n+1}}}}

% \\   \zeta_{n+1}  \ar@{->}[r]^{K_{n+1}}  & \eta_{n+1}

\end{lemma}

\begin{proof}

Since $\gcd(a_{n+1},q_{n+1})=1$, then by lemma \ref{lempermutation}, $K_{n+1}$ is an isomorphism. Moreover, since $q_n$ divides $q_{n+1}$, then $\zeta_n \hookrightarrow  \zeta_{n+1}$. Therefore, we can define the isomorphism $K_n^{n+1}= K_{n+1|\zeta_n}$. Let $\eta_n^{n+1}=K_n^{n+1}(\zeta_n)$. We have $\eta_n^{n+1} \hookrightarrow  \eta_{n+1}$.

It remains to show that $K_n^{n+1}R_{\frac{1}{q_n}} =S_{\frac{a_n}{q_n}} K_n^{n+1}$ (it automatically implies that $\eta_n^{n+1}$ is stable by $S_{\frac{a_n}{q_n}}$, and that there is $C_n^{n+1}: \eta_n \rightarrow \eta_n^{n+1}$ such that $C_n^{n+1} S_{\frac{a_n}{q_n}} =S_{\frac{a_n}{q_n}} C_n^{n+1}$). Let $0 \leq i \leq q_n-1$. We have:

\[ K_n^{n+1}R_{\frac{1}{q_n}} \left( \left[ \frac{i}{q_n},\frac{i+1}{q_n} \right[ \right)= K_{n+1}R_{\frac{1}{q_n}} \left(  \bigcup_{k=0}^{\frac{q_{n+1}}{q_n}-1} \left[ \frac{i}{q_n}+ \frac{k}{q_{n+1}} , \frac{i}{q_n}+ \frac{k+1}{q_{n+1}} \right[ \right) \]

\[ =  K_{n+1}\left(  \bigcup_{k=0}^{\frac{q_{n+1}}{q_n}-1} \left[ \frac{i+1}{q_n}+ \frac{k}{q_{n+1}} , \frac{i+1}{q_n}+ \frac{k+1}{q_{n+1}} \right[ \right)= \bigcup_{k=0}^{\frac{q_{n+1}}{q_n}-1} K_{n+1}\left(   \left[ \frac{1+i}{q_n}+ \frac{k}{q_{n+1}} , \frac{1+i}{q_n}+ \frac{k+1}{q_{n+1}} \right[ \right) \]
 
\[ =  I \times \bigcup_{k=0}^{\frac{q_{n+1}}{q_n}-1}   \left[ \frac{a_{n+1}}{q_n} + \frac{a_{n+1}i}{q_n}+  \frac{a_{n+1}k}{q_{n+1}} , \frac{a_{n+1}}{q_n} + \frac{a_{n+1}i}{q_n}+ \frac{a_{n+1}k}{q_{n+1}}+ \frac{1}{q_{n+1}} \right[  \]

Since $a_{n+1}/q_n=a_n/q_n$ mod 1, we get:

\[ K_n^{n+1}R_{\frac{1}{q_n}} \left( \left[ \frac{i}{q_n},\frac{i+1}{q_n} \right[ \right) = I \times  \bigcup_{k=0}^{\frac{q_{n+1}}{q_n}-1}   \left[ \frac{a_{n}}{q_n} + \frac{a_{n+1}i}{q_n}+  \frac{a_{n+1}k}{q_{n+1}} , \frac{a_{n}}{q_n} + \frac{a_{n+1}i}{q_n}+ \frac{a_{n+1}k}{q_{n+1}}+ \frac{1}{q_{n+1}} \right[ \]

Therefore,

\[ K_n^{n+1}R_{\frac{1}{q_n}} \left( \left[ \frac{i}{q_n},\frac{i+1}{q_n} \right[ \right) =  \bigcup_{k=0}^{\frac{q_{n+1}}{q_n}-1}  S_{\frac{a_{n}}{q_n}} \left( I \times \left[ \frac{a_{n+1}i}{q_n}+  \frac{a_{n+1}k}{q_{n+1}} , \frac{a_{n+1}i}{q_n}+ \frac{a_{n+1}k}{q_{n+1}}+ \frac{1}{q_{n+1}} \right[  \right) \]

\[  =  S_{\frac{a_{n}}{q_n}} \left( I \times \bigcup_{k=0}^{\frac{q_{n+1}}{q_n}-1}  \left[ \frac{a_{n+1}i}{q_n}+  \frac{a_{n+1}k}{q_{n+1}} , \frac{a_{n+1}i}{q_n}+ \frac{a_{n+1}k}{q_{n+1}}+ \frac{1}{q_{n+1}} \right[ \right) =  S_{\frac{a_{n}}{q_n}} K_n^{n+1} \left( \left[ \frac{i}{q_n},\frac{i+1}{q_n} \right[ \right) \]

\end{proof}

Let us denote $R^{(n)}= K_n^{n+1} \left( \left[ 0,\frac{1}{q_n} \right[ \right) $. 

We also denote $R_{i,n}^{n+1}= S_{\frac{ia_n}{q_n}} R^{(n)}, i=0,...,q_n-1$. $R^{(n)}$ is a fundamental domain of $S_{\frac{a_n}{q_n}} $. Moreover, we have:

%\[ R_{i,n}^{n+1}= K_n^{n+1}   \] 

%Let us also denote $C_n^{n+1}$ the map defined by:

\begin{eqnarray*}
C_n^{n+1}\colon \eta_n &\to & \eta_n^{n+1} \\
 \left[ \frac{ia_n}{q_n},\frac{ia_n+1}{q_n} \right[ &\mapsto & R_{i,n}^{n+1}, \, i=0,...,q_{n+1}-1
\end{eqnarray*}
 
Note also that $C_n^{n+1}S_{\frac{a_n}{q_n}}=S_{\frac{a_n}{q_n}}C_n^{n+1}$. Moreover, by assumption \ref{61rotisom} of lemma \ref{lemmefondarotisom}, $a_nb_n/q_n=1/q_n$ mod 1. Therefore, we get: \[ C_n^{n+1}S_{\frac{1}{q_n}}=C_n^{n+1}S_{\frac{a_nb_n}{q_n}}=S_{\frac{a_nb_n}{q_n}}C_n^{n+1} = S_{\frac{1}{q_n}}C_n^{n+1} \]

\bigskip

By iterating lemma \ref{diagcomutrotisom}, we get a corollary that is important for the construction of the isomorphism:

\begin{corollary}
\label{diagitererotisom}
For any $m>n$, there are partitions $\eta_n^{m} \hookrightarrow  \eta_{n+1}^m$ of $M$ such that $\eta_n^{m}$ is stable by $S_{\frac{1}{q_n}}$ and there exists an isomorphism $K_n^{m}: \zeta_n \rightarrow \eta_n^{m}$ such that $K_n^{m} R_{\frac{1}{q_n}}=S_{\frac{a_n}{q_n}} K_n^{m}$ and $K_n^m=K_{n+1|\eta_n^m}^m$.

Said otherwise, the following diagram commutes:

\[ \xymatrix{
\zeta_n^m \ar@(ul,dl)[]_{R_{\frac{1}{q_n}}}  \ar@{->}[r]^{K_n^{m}} \ar@{^{(}->}[d]  & \eta_n^{m} \ar@(ur,dr)[]^{S_{\frac{a_n}{q_n}}} \ar@{^{(}->}[d]  \\
\zeta_{n+1}   \ar@{->}[r]^{K_{n+1}^{m}} & \eta_{n+1}^{m}  
 } \]

%The left hand side of the diagram comes from section \ref{partitiontorus}.\ar@(ul,dl)[]_{R_{\frac{1}{q_{n+1}}}}
%\ar@(ur,dr)[]^{S_{\frac{a_{n+1}}{q_{n+1}}}}
\end{corollary}

\begin{proof} The proof is similar to the one found in the previous chapter.%: since $\eta_n^{n+1} \hookrightarrow \eta_{n+1}$, we can define $C_n^m = C_{m-1}^m...C_{n+1}^{n+2} C_n^{n+1}$. Let $\eta_n^m= C_n^m  \eta_n$. We have: $\eta_{n}^m \hookrightarrow  \eta _{n+1}^m$. We define $K_n^m= K_n C_n^m$. We can check that $K_n^{m} R_{\frac{1}{q_n}}=S_{\frac{a_n}{q_n}} K_n^{m}$ and $K_n^m=K_{n+1|\eta_n^m}^m$.
\end{proof}

For any $n$ fixed, the sequence of partitions $(\eta_n^{m})_{m\geq n}$ must converge when $m \rightarrow + \infty$, in order to obtain a full sequence of monotonic partitions. Moreover, the possible limit sequence (i.e. a possible $\eta_n^\infty$) must generate. Indeed, these assumptions are required to apply lemma \ref{lemmekatokrotisom}. However, we can check that none of these assumptions are satisfied, in general. Therefore, to obtain these assumptions, we pull back the partition $\eta_n^{m}$ by a suitable smooth measure-preserving diffeomorphism $B_m$. The following lemma, already proved in the previous chapter, gives the conditions that $B_m$ must satisfy:

\begin{lemma}
\label{conditionsbnrotisom}
Let $B_m \in$ Diff$^\infty(M,\mu)$. Let $A_{m+1}= B_{m+1} B_m^{-1}$.

\begin{enumerate}
\item If $A_{m+1} S_{\frac{1}{q_m}}=  S_{\frac{1}{q_m} }A_{m+1}$ and if \[ \sum_{m\geq 0} q_m \mu \left( ( I \times [0,q_m[ )\Delta A_{m+1}^{-1} R^{(m)}  \right) < + \infty \]
then for any fixed $n$, when $m \rightarrow + \infty$, the sequence of partitions $\xi_n^m= B_m^{-1} \eta_n^m$ converges. We denote $\xi_n^\infty$ the limit. The sequence $\xi_n^\infty$ is monotonous and $T_n= B_n^{-1} S_{\frac{p_n}{q_n}} B_n$ stabilizes each $\xi_n^\infty$.
\item If, moreover, the sequence $\xi_n= B_n^{-1} \eta_n$ generates, then so does $\xi_n^\infty$.
\end{enumerate}

\end{lemma}
%, the reason why we cannot take $M=\varmathbb{T}^1$
$C_m^{m+1}$ is not continuous in general, and $A_{m+1}$ is its differentiable approximation. Lemma \ref{conditionsbnrotisom} is the reason why we need for $M$ a manifold of dimension $d \geq 2$. Indeed, if we took $M=\varmathbb{T}^1$, we could not find a diffeomorphism $B_m$ satisfying the assumptions of this lemma, except for $a_n=1$ or $a_n=q_n-1$. The choice $a_n=1$ gives that the rotation $R_\alpha$ on the circle is isomorphic to itself. The choice $a_n=q_n-1$ gives that $R_\alpha$ is isomorphic to $R_{-\alpha}$. The existence of these two isomorphisms are consistent with the fact, mentioned in the introduction, that $R_\alpha$ and $R_\beta$ are isomorphic, with $\alpha$ irrational, if and only if $\alpha= \pm \beta$.

\bigskip
%\begin{proof}
%See \cite{katokant}.

%\end{proof}

By adding to lemma \ref{conditionsbnrotisom} the convergence of the sequence $T_n$, we obtain the required isomorphism:

\begin{corollary}
\label{corolisomrotisom}
If both conditions 1. and 2. of lemma \ref{conditionsbnrotisom} hold, and if $T_n= B_n^{-1} S_{\frac{p_n}{q_n}} B_n$ weakly converges towards an automorphism $T$, then $(\varmathbb{T}^1, R_{\beta}, Leb)$ and $(M,T, \mu)$ are metrically isomorphic.
\end{corollary}

%\begin{corollary}
%\label{corolisomrotisom}
%If all the conditions of lemma \ref{conditionsbnrotisom} obtain, and if $T_n= B_n^{-1} S_{\frac{p_n}{q_n}} B_n$ $C^\infty$-converges towards $T$, then $(\varmathbb{T}^1, R_{\beta}$ and $(M,T)$ are metrically isomorphic.

%\end{corollary}

\begin{proof}

By corollary \ref{diagitererotisom}, $K_n^{m} R_{\frac{1}{q_n}}=S_{\frac{a_n}{q_n}} K_n^{m}$. By iteration, 

\[ K_n^{m} R_{\frac{b_np_n}{q_n}}=S_{\frac{a_nb_np_n}{q_n}} K_n^{m} \]

Since $a_nb_n/q_n=1/q_n$ mod 1, then 

%$K_n^{m} T^{ \frac{\gn p_n}{q_n}}=S_{\frac{p_n}{q_n}} K_n^{m}$. 

\[ K_n^{m} R_{\frac{b_np_n}{q_n}}=S_{\frac{p_n}{q_n}} K_n^{m} \]

Therefore, the following diagram commutes:

\[ \xymatrix{
\zeta_n \ar@(ul,dl)[]_{R_{\frac{b_np_n}{q_n}} }  \ar@{->}[r]^{K_n^{m}} \ar@{^{(}->}[d]  & \eta_n^{m} \ar@(ul,ur)[]^{S_{\frac{p_n}{q_n}}}  \ar@{->}[r]^{B_m^{-1}} \ar@{^{(}->}[d]  & \xi_n^{m} \ar@(ur,dr)[]^{T_n} \ar@{^{(}->}[d]   \\
\zeta_{n+1}   \ar@{->}[r]^{K_{n+1}^{m}} & \eta_{n+1}^{m}  \ar@{->}[r]^{B_m^{-1}} & \xi_{n+1}^{m} 
 } \]
 
% \ar@(ur,dr)[]^{T_{n+1}}\ar@(ul,dl)[]_{R_{\frac{b_{n+1}p_{n+1}}{q_{n+1}}}}\ar@(dl,dr)[]_{S_{\frac{p_{n+1}}{q_{n+1}}}}
 
Let $\bar{K}_n^\infty: \zeta_n \rightarrow \xi_n^\infty $ be defined by $\bar{K}_n^\infty= P_n^\infty B_n^{-1} K_n $, where $P_n^\infty : \xi_n \rightarrow \xi_n^\infty$ is the limit isomorphism of the sequence $P_n^m=   B_m^{-1} C_n^m B_n$, defined in the same way as in the previous chapter.  

As in the proof of corollary 3.4 in the previous chapter, we can show that $\bar{K}_n^\infty R_{b_n \frac{p_n}{q_n}} = T_n \bar{K}_n^\infty$ and that $\bar{K}_{n+1|\zeta_n}^\infty= \bar{K}_n^\infty$.
This allows to apply lemma \ref{lemmekatokrotisom}, which gives the required metric isomorphism.

%\ar@(ur,dr)[]^{S_{\frac{1}{q_n}}} \ar@(ur,dr)[]^{S_{\frac{1}{q_{n+1}}}} 
%\[ \xymatrix{
%\zeta_n^m \ar@(ul,dl)[]_{T^{\frac{\gn}{q_n}}}  \ar@{->}[r]^{K_n^{m}} \ar@{^{(}->}[d]  & \eta_n^{m}  \ar@{^{(}->}[d] \ar@{->}[r]^{B_m^{-1}} &  B_m^{-1} \eta_n^{m} \ar@(ur,dr)[r]^{T_n} \ar@{^{(}->}[d] \\
%\zeta_{n+1}^m \ar@(ul,dl)[]_{T^{\frac{\gnu}{q_{n+1}}}}  \ar@{->}[r]^{K_{n+1}^{m}} & \eta_{n+1}^{m} 
% }\ar@{->}[r]^{B_m^{-1}} &  B_m^{-1} \eta_n^{m} \ar@(ur,dr)[r]^{T_n}  \]
%then $(\varmathbb{T}^h, T^\alpha)$ and $(M,T)$ are metrically isomorphic.

%We apply lemma \ref{conditionsbnrotisom}, and lemma \ref{lemmekatokrot} gives the required metric isomorphism.
\end{proof}

Let us make one remark. We consider the isomorphism between $R_{\frac{b_np_n}{q_n}}$ and 

$T_n= B_n^{-1} S_{\frac{p_n}{q_n}} B_n$, instead of the isomorphism between $R_{\frac{p_n}{q_n}}$ and $\check{T}_n= B_n^{-1} S_{\frac{a_n p_n}{q_n}} B_n$ (which seems to be a more "natural" choice), because in the latter case, we are not able to show the convergence of $\check{T}_n$ towards a smooth diffeomorphism $\check{T}$. Indeed, we have:

\begin{equation}
\label{estbidon}
d_k(\check{T}_{n+1},\check{T}_{n}) \leq \| B_{n+1} \|_k \left| \frac{p_{n+1}a_{n+1}}{q_{n+1}} - \frac{p_{n}a_{n}}{q_{n}} \right|_{\mod 1} = \| B_{n+1} \|_k  a_{n+1} \left| \frac{p_{n+1}}{q_{n+1}} - \frac{p_{n}}{q_{n}} \right| 
\end{equation}

In the next section, we show that $\| B_{n+1} \|_k \leq (b_{n+1}q_n)^{R_{\ref{r0csterotisom}}(n)}$ for some fixed sequence $R_{\ref{r0csterotisom}}(n)$ (and we are not able to improve this estimate). Estimate (\ref{estbidon}) becomes:

%Estimate (\ref{estbidon}) becomes:, element of the partition $\eta_n^{n+1}$

\[ d_k(\check{T}_{n+1},\check{T}_n) \leq  a_{n+1} (b_{n+1}q_n)^{R_{\ref{r0csterotisom}}(n)} \left| \frac{p_{n+1}}{q_{n+1}} - \frac{p_{n}}{q_{n}} \right| \]

Moreover, by assumption 1 of lemma \ref{lemmefondarotisom}, $a_{n+1}b_{n+1} \geq q_{n+1}$. Since  $\left| \frac{p_{n+1}}{q_{n+1}} - \frac{p_{n}}{q_{n}} \right| \geq \frac{1}{q_{n+1}} $, estimate (\ref{estbidon}) does not allow to show that $\check{T}_n$ is Cauchy. On the other hand, applying this reasoning to show the convergence of $T_n$ will be successful.

\bigskip

In order to construct the diffeomorphism $A_{n+1}$ with suitable estimates of its norm, we need to control the width of the connected components of $R^{(n)}$. A priori, $R^{(n)}$ consists of $q_{n+1}/q_n$ "slices" of width $1/q_{n+1}$. However, this fact does not ensure the convergence of $T_n$, because it only implies that $\|B_{n+1} \|_{j} \leq (q_{n+1})^{R(n)}$ for some fixed sequence of integers $R(n)$. In order to apply the reasoning above successfully, we need a better estimate. The following lemma shows that "slices" of $R^{(n)}$ of width $1/q_{n+1}$ stack on each other, which gives $b_{n+1}$ connected components to $R^{(n)}$, each having a width of order $1/(q_n b_{n+1})$. This will allow an estimate of the form $\|B_{n+1} \|_{j} \leq (q_n b_{n+1})^{R(n)}$, which will ensure the convergence of $T_n$, by taking $b_{n+1}$ small. %phenomenon is critical, because the former requires that the series $\sum_{n \geq 0} \| B_{n+1} \|_k \left| \frac{p_{n+1}}{q_{n+1}} - \frac{p_{n}}{q_{n}} \right|$ is convergent. We have the lemma: 

\begin{lemma}
\label{estdern}

Let \[ m_n= \frac{q_{n+1}}{q_n} -1 - b_{n+1} \left\lfloor \frac{\frac{q_{n+1}}{q_n} -1}{b_{n+1}} \right\rfloor \]

and for $0 \leq l \leq b_{n+1}-1$, let

\[ k_n(l)= \left\lfloor l a_{n+1} \frac{q_n}{q_{n+1}} \right\rfloor \]

\[ r_n(l)= la_{n+1}- \frac{q_{n+1}}{q_n} k_n(l) \]

We have: 

\[ R^{(n)}= \bigcup_{l=0}^{b_{n+1}-1}  R^{(n),l} \]

with, if $0 \leq l \leq m_n$:

\[  R^{(n),l} = I \times \left( \frac{k_n(l)}{q_n}+ \frac{r_n(l)}{q_{n+1}} + \left[ 0, \frac{\left\lfloor \frac{\frac{q_{n+1}}{q_n} -1}{b_{n+1}} \right\rfloor +1}{q_{n+1}} \right[ \right) \]

and if $m_n+1 \leq l \leq b_{n+1}-1$:

\[  R^{(n),l} = I \times \left( \frac{k_n(l)}{q_n}+ \frac{r_n(l)}{q_{n+1}} + \left[ 0, \frac{\left\lfloor \frac{\frac{q_{n+1}}{q_n} -1}{b_{n+1}} \right\rfloor }{q_{n+1}} \right[ \right)  \]

\end{lemma}

\begin{figure}[h]
\label{rngraphique}
\centering
\includegraphics[height=8cm]{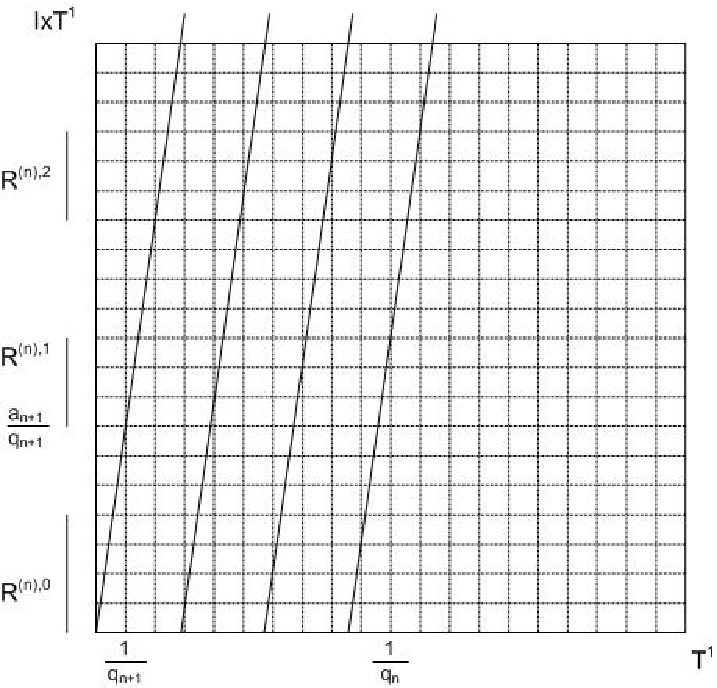}
\caption{The set $R^{(n)}=K_n^{n+1} (I \times [0,1/q_n[)$ for $q_n=2$, $q_{n+1}=20$, $a_{n+1}=7$, $b_{n+1}=3$. $R^{(n)}=R^{(n),0} \cup R^{(n),1} \cup R^{(n),2}$ has $b_{n+1}=3$ connected components. The oblique lines represent the graph of the map $x \mapsto a_{n+1} x$ from $\varmathbb{T}^1$ to itself. In this illustration, $k_n(0)=k_n(1)=0$, $k_n(2)=1$, $r_n(0)=0$, $r_n(1)=7$, $r_n(2)=4$. }
\end{figure}

%Figure 1 represents the set $R^{(n)}$. 

\begin{proof}

We have:

\[ R^{(n)}= I \times \bigcup_{i=0}^{q_{n+1}/q_n-1} \left[  \frac{a_{n+1}i}{q_{n+1}},  \frac{a_{n+1}i}{q_{n+1}} + \frac{1}{q_{n+1}}  \right[   \]

For $i=0,...,q_{n+1}/q_n-1$, we make the Euclidean division of $i$ by $b_{n+1}$. We get: \[ i= k_ib_{n+1}+r_i \]

with $0 \leq r_i \leq b_{n+1}-1$ and $0 \leq k_i \leq \left\lfloor \frac{\frac{q_{n+1}}{q_n} -1}{b_{n+1}} \right\rfloor$. Since $a_{n+1}b_{n+1}/q_{n+1}=1/q_{n+1}$ mod 1, we get:

\[ R^{(n)}=I \times \bigcup_{i=0}^{q_{n+1}/q_n-1} \left[  \frac{a_{n+1}r_i+k_i}{q_{n+1}},  \frac{a_{n+1}r_i+k_i}{q_{n+1}} + \frac{1}{q_{n+1}}  \right[   \]

Moreover, we have:

\[ \{ 0,...,q_{n+1}/q_n-1 \}= \{ 0,...,b_{n+1}-1\} \bigcup \left( b_{n+1} + \{ 0,...,b_{n+1}-1\} \right) \bigcup ...\] 

\[ ...\bigcup \left( b_{n+1} \left( \left\lfloor \left(\frac{q_{n+1}}{q_n} -1\right)/b_{n+1} \right\rfloor-1 \right) + \{ 0,...,b_{n+1}-1\} \right) \bigcup \left( b_{n+1} \left( \left\lfloor \left(\frac{q_{n+1}}{q_n} -1\right)/b_{n+1} \right\rfloor \right)+ \{ 0,...,m_n \} \right)  \]

Therefore,

\[ R^{(n)}= I \times \bigcup_{k_i=0}^{\left\lfloor (\frac{q_{n+1}}{q_n} -1)/b_{n+1} \right\rfloor-1} \bigcup_{r_i=0}^{b_{n+1}-1} \left[  \frac{a_{n+1}r_i+k_i}{q_{n+1}},  \frac{a_{n+1}r_i+k_i}{q_{n+1}} + \frac{1}{q_{n+1}}  \right[   \bigcup \]

\[ \bigcup_{r_i=0}^{m_n} \left[  \frac{a_{n+1}r_i+\left\lfloor (\frac{q_{n+1}}{q_n} -1)/b_{n+1} \right\rfloor}{q_{n+1}},  \frac{a_{n+1}r_i+\left\lfloor (\frac{q_{n+1}}{q_n} -1)/b_{n+1} \right\rfloor}{q_{n+1}} + \frac{1}{q_{n+1}}  \right[  \]

\[ R^{(n)}= I \times \bigcup_{l=0}^{m_n} \left( \frac{a_{n+1}l}{q_{n+1}}+ \bigcup_{k_i=0}^{\left\lfloor (\frac{q_{n+1}}{q_n} -1)/b_{n+1} \right\rfloor}  \left[  \frac{k_i}{q_{n+1}},  \frac{k_i}{q_{n+1}} + \frac{1}{q_{n+1}}  \right[   \right)  \bigcup \]

\[ \bigcup_{l=m_n+1}^{b_{n+1}-1} \left( \frac{a_{n+1}l}{q_{n+1}}+ \bigcup_{k_i=0}^{\left\lfloor (\frac{q_{n+1}}{q_n} -1)/b_{n+1} \right\rfloor-1} \left[  \frac{k_i}{q_{n+1}},  \frac{k_i}{q_{n+1}} + \frac{1}{q_{n+1}}  \right[ \right)  \]

\[ R^{(n)}= I \times \bigcup_{l=0}^{m_n} \left( \frac{a_{n+1}l}{q_{n+1}}+  \left[  0,  \frac{\left\lfloor (\frac{q_{n+1}}{q_n} -1)/b_{n+1} \right\rfloor+1}{q_{n+1}}  \right[   \right) \bigcup_{l=m_n+1}^{b_{n+1}-1} \left( \frac{a_{n+1}l}{q_{n+1}}+ \left[ 0,  \frac{\left\lfloor (\frac{q_{n+1}}{q_n} -1)/b_{n+1} \right\rfloor}{q_{n+1}}  \right[ \right)  \]

Finally, the Euclidean division of $la_{n+1}$ by $q_{n+1}/q_n$ gives: 

\[ la_{n+1}= k_n(l)  q_{n+1}/q_n + r_n(l) \]

with 
\[ k_n(l)= \left\lfloor l a_{n+1} \frac{q_n}{q_{n+1}} \right\rfloor \]

\[ r_n(l)= la_{n+1}- \frac{q_{n+1}}{q_n} \left\lfloor l a_{n+1} \frac{q_n}{q_{n+1}} \right\rfloor \]

We get: 

\[ R^{(n)}=I \times  \bigcup_{l=0}^{m_n} \frac{k_n(l)}{q_n}+ \frac{r_n(l)}{q_{n+1}} + \left[ 0, \frac{\left\lfloor \frac{\frac{q_{n+1}}{q_n} -1}{b_{n+1}} \right\rfloor +1}{q_{n+1}} \right[  \bigcup_{l= m_n+1}^{b_{n+1}-1} \frac{k_n(l)}{q_n}+ \frac{r_n(l)}{q_{n+1}} + \left[ 0, \frac{\left\lfloor \frac{\frac{q_{n+1}}{q_n} -1}{b_{n+1}} \right\rfloor }{q_{n+1}} \right[  \]

\end{proof}

\bigskip

The next section is dedicated to the construction of the sequence of diffeomorphisms $B_n$ satisfying the conditions of lemma \ref{conditionsbnrotisom}.

\section{The sequence of conjugacies}

In this section, we construct a sequence of diffeomorphisms $B_n$ on $M$ satisfying the conditions of lemma \ref{conditionsbnrotisom}, such that $ \|B_{n} \|_{n} \leq (q_{n-1} b_{n})^{R_{\ref{rdebnrotisom}}(n-1)}$ for some $R_{\ref{rdebnrotisom}}(n)$, and such that $B_n=Id$ on a neighborhood of $\ddd M$, in order to ensure that $(D^jT)_{|\ddd M}= (D^j S_\alpha)_{|\ddd M}$ for any $j \in \varmathbb{N}$.%  conjugating the rotation $S_{\frac{p_n}{q_n}}$ to $T_n$. $B_n= A_n...A_1$ must satisfy the following specifications:

\begin{proposition}
\label{propexistencerotisom}
There exists a sequence of diffeomorphisms $B_n \in$ Diff$^\infty(M,\mu)$ such that $B_n$ and $A_{n+1}= B_{n+1} B_n^{-1}$ satisfy the following conditions:

\begin{enumerate}

\item (convergence of the partition $\xi_n^m$ to $\xi_n^\infty$) \[ \sum_{ m\geq 0} q_m \mu \left( \Delta_{0,q_m} \Delta A_{m+1}^{-1} R^{(m)}  \right) < + \infty \]
\item (generation) There is a set $E_{n+1} \subset M$ such that $\sum_{n \geq 0} \mu(E_{n+1}^c)< + \infty$ and such that  \[ diam\left( A_{n+1}^{-1} \left( \Delta_{0,q_{n+1}} \bigcap E_{n+1} \right)  \right) \leq \frac{1}{2^n \|B_n\|_1} \] 
\item (equivariance) \[ A_{n+1} S_{\frac{1}{q_n}}=  S_{\frac{1}{q_n} }A_{n+1} \]
\item (polynomial estimation) There is a fixed sequence $\rrr(n) \label{rdebnrotisom} \in \varmathbb{N}$ such that \[ \|A_{n+1} \|_{n+1} \leq (q_n b_{n+1})^{R_{\ref{rdebnrotisom}}(n)} \]
\item (identity on a neighborhood of the boundary) $B_n=Id$ on a neighborhood of $\ddd M$.
\end{enumerate}

\end{proposition}

\begin{remark}
Specification 2 above implies that $\xi_n$ generates (and so $\xi_n^\infty$, by lemma \ref{conditionsbnrotisom}), see \cite{katokant} and \cite{windsor07}.
\end{remark}

We construct $B_n$ recursively. We suppose that $B_n$ exists and satisfies these specifications, and we construct $A_{n+1}$.

The diffeomorphism $A_{n+1}$ is constructed in three steps, each step giving a smooth, measure-preserving, equivariant and polynomially controlled map. In the first step, lemma \ref{lemme1rotisom}, we construct a smooth map $A_{n+1}^1$ that "quasi-cuts" $I \times [0,1/q_{n}[ $ into $b_{n+1}$ vertical slices, and then rotates each slice $\Gamma_l$ by an angle $k_n(l)/q_n$ along the periodic flow $S_t$. %(remember that the parameters $0 \leq a_n(i) \leq q_n-1$ are defined by $R^{(n)}=\bigcup_{i=0}^{k_n-1} \Delta_{a_n(i)k_n+i,k_n q_{n}} $).

%the horizontal slice $ \frac{k_n(l)}{q_n}+ \frac{r_n(l)}{q_{n+1}} + \left[ 0, \frac{\left\lfloor \frac{\frac{q_{n+1}}{q_n} -1}{b_{n+1}} \right\rfloor +1}{q_{n+1}} \right[ \times I$ or $\frac{k_n(l)}{q_n}+ \frac{r_n(l)}{q_{n+1}} + \left[ 0, \frac{\left\lfloor \frac{\frac{q_{n+1}}{q_n} -1}{b_{n+1}} \right\rfloor +1}{q_{n+1}} \right[ \times I$, depending on whether $l \leq m_n$ or $l>m_n$.

In the second step, we construct a second map $A_{n+1}^2$ that "quasi-sends" each vertical slice $A_{n+1}^1(\Gamma_l)$ into a suitable connected component of $R^{(n)}$ (see lemma \ref{estdern} for the decomposition of $R^{(n)}$ into connected components). These two steps ensure that $A_{n+1}$ "quasi-sends" $I \times [0,1/q_{n}[$ to $R^{(n)}$. It ensures that $\xi_n$ converges.

In the third step, we obtain the generation of $\xi_n$. We use $A_{n+1}^3$ to quasi-rotate slices inside each connected component of $R^{(n)}$. These slices are chosen sufficiently thin to ensure that the diameter of $A_{n+1}^{-1} \left( I \times [l/q_{n+1},(l+1)/q_{n+1}[ \right)$ is small, but these slices are not too thin to ensure that $\|A_{n+1} \|_{n+1} \leq (b_{n+1} q_n)^{R_{\ref{rdebnrotisom}}(n)}$, which enables the convergence of $T_n$. This last step completes the construction.

\bigskip

%For the sake of readability, we first write the construction in the case $M= [0,1] \times \varmathbb{T}$, and then we extend it to $M= [0,1]^{d-1} \times \varmathbb{T}$.

Let $l_0,...,l_{b_{n+1}-1}$ integers such that $0 = r_n(l_0)<...<r_n(l_{b_{n+1}-1}) \leq q_{n+1}/q_n-1$. Let $l_{b_{n+1}}=b_{n+1}$ and $r_n(l_{b_{n+1}})=q_{n+1}/q_n$.

\bigskip

%For $i=0,...,b_{n+1}-1$,

\subsection{Construction in dimension 2}

We suppose $M= [0,1] \times \varmathbb{T}^1$. The first step is based on the following lemma, which analogous is found in the previous chapter:% (see figure \ref{step1fig}):

\begin{lemma}
\label{lemme1rotisom}
Let $\frac{1}{b_{n+1}} > \epsilon_1>0$, and $ \G_i= \left[r_n(l_i)\frac{q_n}{q_{n+1}},r_n(l_{i+1})\frac{q_n}{q_{n+1}} -\epsilon_1 \right] \times \left[0, \frac{1}{q_n} \right[ $ for $0 \leq i \leq b_{n+1}-1$. There is a smooth measure-preserving diffeomorphism $A_{n+1}^1$ of $[0,1] \times \varmathbb{T}^1$ such that:

\begin{enumerate}
\item \[ A_{n+1}^1 S_{\frac{1}{q_n}} = S_{\frac{1}{q_n}} A_{n+1}^1 \]
\item \[ A_{n+1}^1 (\G_i)= S_{\frac{k_n(l_i)}{q_n}} \G_i \]
\item \[  \| A_{n+1}^1 \|_j \leq \frac{1}{\epsilon_1^j} \| \phi \|_j    \]
where $\phi$ is a fixed smooth diffeomorphism, independent of $n$ and $\epsilon_1$.
\end{enumerate}

\end{lemma}

%\begin{figure}
%\centering
%\includegraphics[height=8cm]{step1new}
%\caption{The image of the partition $\eta_n$ after step 1.}
%\label{step1fig}
%\end{figure}

We take $\epsilon_1=\frac{1}{ 2^n b_{n+1}}$ and we let: \[ E_{n+1}^1=   \bigcup_{i=0}^{b_{n+1}-1} \left[ r(l_i) \frac{q_n}{q_{n+1}}, r(l_{i+1}) \frac{q_n}{q_{n+1}} - \epsilon_1 \right] \times \varmathbb{T}^1 \]

We have: \[ \mu \left( (E_{n+1}^{1})^c \right) = b_{n+1} \epsilon_1 = \frac{1}{2^n}  \]

\bigskip

In the second step, we shrink $A_{n+1}^1(\G_i)$ horizontally by a factor $q_n$, we expand it vertically by the same factor, and we rotate it by a $\pi/2$ angle (except in a neighborhood of the border of $I \times[0,1/q_n[$). Thus, $A_{n+1}^1(\G_i)$ is quasi-sent to a connected component of  $R^{(n)}$.% Note that with this operation, we automatically obtain the right combinatorics of $R^{(n)}$. 

%This step is slightly different than in Anosov-Katok's original paper, but this difference is critical: following Anosov-Katok's original method, we would need to quasi-permute $k_n q_n$ slices, in order to match the location of $R^{(n)}$. This would require at least $q_n$ iterations, thus jeopardizing the polynomial estimation in $q_n$, and ultimately the obtention of all Liouville numbers on the $h^{th}$ coordinate. 

We have the lemma (see the previous chapter):

\begin{lemma}
\label{lemme2rotisom}
For $0 \leq i \leq b_{n+1}-1$, let $\G'_i= \left[r_n(l_i)\frac{q_n}{q_{n+1}},r_n(l_{i+1})\frac{q_n}{q_{n+1}} \right] \times [0, \frac{1}{q_n}]$. There exists a smooth measure-preserving diffeomorphism $ A_{n+1}^2$ of $[0,1] \times \varmathbb{T}^1$, equivariant by $S_{\frac{1}{q_n}} $ and a measurable set $E_{n+1}^2$ that is globally invariant by $S_{\frac{1}{q_n}} $ and  $ A_{n+1}^2$ such that: \[   A_{n+1}^2  \left( \G'_i \bigcap E_{n+1}^2 \right)=  I \times  \left[\frac{r_n(l_i)}{q_{n+1}},\frac{r_n(l_{i+1})}{q_{n+1}} \right] \bigcap E_{n+1}^2  \]

Moreover, there is an explicit function $R_2(j)$, depending only on $j$, such that: \[ \| A_{n+1}^2 \|_j \leq (q_n )^{R_2(j)} \] %]|\phi_n \|_j  \]

and such that if $\Gamma \subset M$ with $ diam (\Gamma) \leq x$, then: \[ diam \left( (A_{n+1}^2)^{-1}  \left( \Gamma \bigcap E_{n+1}^2 \right) \right) \leq q_n x \]

\end{lemma}

%\begin{figure}[h]
%\centering
%\includegraphics[height=5cm]{avantquasirot}
%\caption{The partition $A_{n+1}^1 \eta_n \cap [0,1] \times \left[\frac{i}{q_n},\frac{i+1}{q_n}\right]$ before step 2 ($\epsilon_1$ has been taken infinitesimally small in the illustration).}
%\label{avantquasirotfig}
%\end{figure}

%\begin{figure}
%\centering
%\includegraphics[height=5cm]{apresquasirot}
%\caption{The partition $A_{n+1}^1 \eta_n \cap [0,1] \times \left[\frac{i}{q_n},\frac{i+1}{q_n}\right]$ after step 2 ($\epsilon_1$ has been taken infinitesimally small in the illustration).}
%\label{apresquasirotfig}
%\end{figure}

Combined with lemma \ref{lemme1rotisom}, lemma \ref{lemme2rotisom} gives the following corollary, which implies the convergence of the partition $\xi_n$ to $\xi_n^\infty$ (see the previous chapter):

\begin{corollary}
\label{cgcepartioncorrotisom}
We have the estimation: 

\[ \mu \left( A_{n+1}^2 A_{n+1}^1 \left(I  \times  \left[ 0,1/q_n \right[  \right)  \Delta R^{(n)}  \right)  \leq \frac{8}{2^n q_n} \]

\end{corollary}

%We now prove lemma \ref{lemme2rotisom}:
%\begin{proof}[Proof of lemma \ref{lemme2rotisom}.]

\begin{remark}
The construction above also shows that there exists a permutation $\sigma$ of $\{ R^{(n),l}+ j/q_n, l=0,...,b_{n+1}-1,j=0,...,q_n-1 \}$ such that for any $l=0,...,b_{n+1}-1,j=0,...,q_n-1$,

\[  \mu\left(A_{n+1}^2 A_{n+1}^1 (A_{n+1}^2)^{-1} ( I \times R^{(n),l}+ j/q_n ) \Delta I \times \sigma\left( R^{(n),l}+ j/q_n \right) \right) \leq  \frac{16}{2^nq_n}   \]

This observation is generalized in \cite{nlbernoulli} and used in \cite{gausskronecker}, where conjugacies correspond to smooth approximations of permutations.
\end{remark}

In the third step, it remains to obtain the generation of the sequence of partitions $\xi_n^\infty$, without affecting the properties obtained in the first two steps. In particular, we cannot exactly proceed as in the previous chapter, because $b_{n+1}$ can be a bounded sequence (while in the previous chapter, the sequence $k_{n}$, analogous to $b_{n+1}$, is larger than $q_n$). We need to refine the approach of the previous chapter. This third step is based on the following lemma and its corollary:

\begin{lemma}
\label{lemme3}
For any integer $w \geq 1$, there is a smooth, measure-preserving, and $S_{\frac{1}{ q_n}}$-equivariant diffeomorphism $A_{n+1}^3$, and an explicit sequence of integers $\rrr(n)\label{estdeanplusuntrois}$, such that:

\[ \|A_{n+1}^3 \|_{n+1}  \leq  \left(\frac{q_{n+1}}{w}\right)^{R_{\ref{estdeanplusuntrois}}(n)} \]%\|\phi_n \|_j  \] 

and there exists a $S_{\frac{1}{ q_n}}$-invariant and $A_{n+1}^3$-invariant set $E_{n+1}^3$ such that 

$\mu( E_{n+1}^{3 c}) \leq 4/2^n$, and such that for any $i=0,...,q_{n+1}-1$, we have: 

\[  diam \left( (A_{n+1}^3)^{-1} \left(  \left[0,1 \right] \times \left[ i/q_{n+1} ,  (i+1)/q_{n+1} \right[  \right) \cap E_{n+1}^{3} \right) \leq \max \left(  \frac{1}{w}, \frac{2w}{q_{n+1}} \right)  \]

\end{lemma}

%Note that in particular, $A_{n+1}^3$ leaves invariant each $R_{i,n}^{n+1}$, thus not affecting the work done in the first two steps. If we 
%we obtain the conditions of proposition \ref{borelcantelli}, which imply generation:

We obtain the corollary:

\begin{corollary}
\label{corogenerationrotisom}
There exists an explicit sequence of integers $\rrr (n) \label{rpol3}$ depending only on $n$, there is a smooth, measure-preserving, and $S_{\frac{1}{ q_n}}$-equivariant diffeomorphism $A_{n+1}$, such that:

\[ \|A_{n+1} \|_{n+1}  \leq  \left(b_{n+1} q_n \right)^{R_{\ref{rpol3}}(n)} \] %\|\phi_n \|_{n+1}  \] 

and there exists a $S_{\frac{1}{ q_n}}$-invariant and $A_{n+1}$-invariant set $E_{n+1}$ such that $\mu( E_{n+1}^{ c}) \leq 4/2^n$, and such that for any $i=0,...,q_{n+1}-1$, we have: 

\[  diam \left( (A_{n+1})^{-1} \left(\left[0,1 \right] \times  \left[ i/q_{n+1} ,  (i+1)/q_{n+1} \right[  \right) \cap E_{n+1} \right) \leq \frac{1}{2^n \|B_n\|_1}   \]

\end{corollary}

\begin{figure}[h]
\label{anplusuntroisfig}
\centering
\includegraphics[height=8cm]{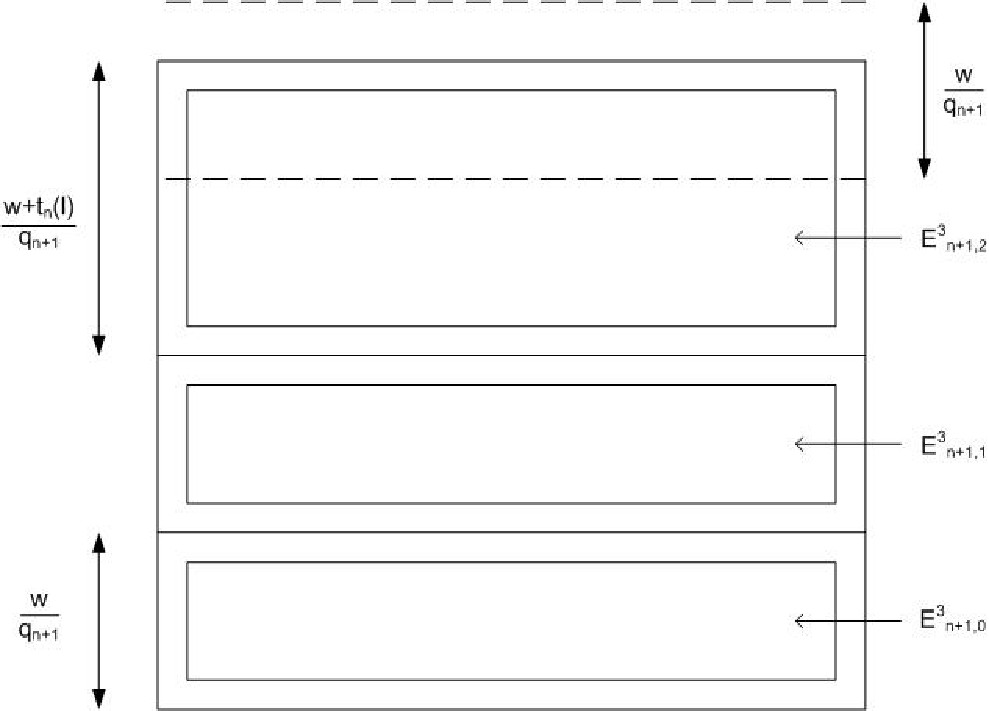}
\caption{Illustration of the third step (generation): a connected component of $R^{(n)}$ with $h_n(l)=3$.}
\end{figure}

\begin{proof}[Proof of lemma \ref{lemme3}.]

We define $A_{n+1}^3$ on $R^{(n)}$, and since $R^{(n)}$ is a fundamental domain of $S_{\frac{1}{q_n}}$, we can extend it to all $M$ by $S_{\frac{1}{q_n}}$-periodicity. To that aim, we define $A_{n+1}^3$ on each connected component of $R^{(n)}$ (see figure 16). %\ref{anplusuntroisfig}

Let $f_n(l)= \left\lfloor \frac{\frac{q_{n+1}}{q_n} -1}{b_{n+1}} \right\rfloor +1$ if $0 \leq l \leq m_n$ and $f_n(l)=\left\lfloor \frac{\frac{q_{n+1}}{q_n} -1}{b_{n+1}} \right\rfloor$ if $ m_n+1 \leq l \leq b_{n+1}-1$ ($f_n(l)/q_{n+1}$ is the width of a connected component of $R^{(n)}$, see lemma \ref{estdern}). We perform the Euclidean division of $f_n(l)$ by $w$: \[f_n(l)= h_n(l)w+t_n(l) \] with $0 \leq t_n(l) \leq w-1$.

We also need to recall the definition of a "quasi-rotation" by $\pi/2$ \cite{windsor07}:

\begin{proposition}
\label{propwindsorrotisom}
For any $n \geq 1$, there is a smooth measure-preserving map 

$\phi_n: [0,1]^2 \rightarrow [0,1]^2$ (called "quasi-rotation") such that $\phi_n = R_{\pi/2}$ on $[\frac{1}{2^n}, 1- \frac{1}{2^n}]^2$ and $\phi_n = Id$ on a neighborhood of the boundary of $[0,1]^2$.

\end{proposition}

Let $p \geq 2$ a real number and \[ C_p : \begin{array}[t]{lcl} [0,1] \times \left[0, \frac{1}{p}\right]  &\rightarrow &   [0,1] \times [0,1]  \\
                               (x,y) & \mapsto    & (x,py)
           \end{array}
           \]

Let $\phi_{n,p}= C_p^{-1} \phi_n C_p$. The map $\phi_{n,p}$ is smooth and measure-preserving. By the Faa-di-Bruno formula, there exists a fixed function $ \rrr(j) \label{r1jrotisom}$ such that \[ \| \phi_{n,p} \|_j  \leq p^{R_{\ref{r1jrotisom}}(j)} \|\phi_n \|_j \]

Since $\phi_n$ is fixed, by choosing a larger $R_{\ref{r1jrotisom}}(n)$, we have:

\[ \| \phi_{n,p} \|_{n+1}  \leq p^{R_{\ref{r1jrotisom}}(n)} \]

For $0 \leq l \leq b_{n+1}$, on $[0,1] \times \left( \frac{k_n(l)}{q_n}+ \frac{r_n(l)}{q_{n+1}} + \left[ 0, \frac{w}{q_{n+1}} \right] \right) $, we let $A_{n+1}^3 = \phi_{n,q_{n+1}/w} $ and \[ E_{n+1,0}^3= \left[\frac{1}{2^{n+1}},1-\frac{1}{2^{n+1}} \right]  \times  \left( \frac{k_n(l)}{q_n}+ \frac{r_n(l)}{q_{n+1}} + \left[\frac{w}{2^{n+1}q_{n+1}},  \frac{w}{q_{n+1}} - \frac{w}{2^{n+1}q_{n+1}} \right] \right) \] We extend $A_{n+1}^3$ to $[0,1] \times \left( \frac{k_n(l)}{q_n}+ \frac{r_n(l)}{q_{n+1}} + \left[ 0, \frac{(h_n(l)-1)w}{q_{n+1}} \right] \right) $ by $S_{\frac{w}{q_{n+1}}}$-equivariance. 

\bigskip

Likewise, for $x=1,...,h_n(l)-2$, we define: $ E_{n+1,x}^3= \frac{xw}{q_{n+1}} +  E_{n+1,0}^3 $.

\bigskip

On $ [0,1] \times  \left( \frac{k_n(l)}{q_n}+ \frac{r_n(l)}{q_{n+1}} + \left[ \frac{(h_n(l)-1)w}{q_{n+1}}, \frac{h_n(l)w+t_n(l)}{q_{n+1}} \right] \right) $, we let $A_{n+1}^3 =  \phi_{n,\frac{q_{n+1}}{w+t_n(l)}} $. This completes the construction of $A_{n+1}^3$ on $R^{(n)}$. By $S_{\frac{1}{q_{n}}}$-equivariance, we get the definition of $A_{n+1}^3$ on the whole manifold $M$. Moreover, since $0 \leq t_n(l) \leq w-1$, there exists $R_{\ref{estdeanplusuntrois}}(n)$ such that:

\[ \|A_{n+1}^3 \|_{n+1}  \leq \max \left( \| \phi_{n,\frac{q_{n+1}}{w}} \|_{n+1},  \max_{0 \leq l \leq b_{n+1}-1} \| \phi_{n,\frac{q_{n+1}}{w+t_n(l)}} \|_{n+1} \right) \leq   \left(\frac{q_{n+1}}{w}\right)^{R_{\ref{estdeanplusuntrois}}(n)}   \]

\bigskip

Let: \[ E_{n+1,h_n(l)-1}^3= \left[\frac{1}{2^{n+1}},1-\frac{1}{2^{n+1}} \right] \times \left( \frac{k_n(l)}{q_n}+ \frac{r_n(l)}{q_{n+1}} + \frac{(h_n(l)-1)w}{q_{n+1}}  +  \left[ \frac{w+t_n(l)}{q_{n+1}2^{n+1}}, \frac{w+t_n(l)}{q_{n+1}} - \frac{w+t_n(l)}{q_{n+1}2^{n+1}} \right] \right)  \]

We let $E_{n+1}^{3,0}= \bigcup_{x=0}^{h_n(l)-1} E_{n+1,x}^3$, and $ E_{n+1}^3= \bigcup_{y=0}^{q_n-1} \frac{y}{q_{n}} +  E_{n+1}^{3,0} $. This completes the construction of $ E_{n+1}^3$. Moreover, we have: $\mu( E_{n+1}^{3 c}) \leq 4/2^n$.

\bigskip

Moreover, since $t_n(l) \leq w$, then for any $i=0,...,q_{n+1}-1$, we have: 

\[  diam \left( (A_{n+1}^3)^{-1} \left( \left[ i/q_{n+1} ,  (i+1)/q_{n+1} \right[ \times \left[0,1 \right] \right) \cap E_{n+1}^{3} \right) \leq  \max \left(  \frac{1}{w}, \max_{0 \leq l \leq b_{n+1}-1} \frac{w+t_n(l)}{q_{n+1}} \right)  \leq \max \left(  \frac{1}{w}, \frac{2w}{q_{n+1}} \right)  \]

\end{proof}

%We now prove :

\begin{proof}[Proof of corollary \ref{corogenerationrotisom}.] 
By the recurrence assumption on $B_n$, there exists $\rrr(n) \label{rdebnn}$ such that $\|B_n\|_1 \leq (b_n q_{n-1})^{R_{\ref{rdebnn}}(n-1)}$. Let \[ w = \left\lfloor  \frac{q_{n+1}}{2^{n+1} q_n^2 (b_n q_{n-1})^{R_{\ref{rdebnn}}(n-1)}  }  \right\rfloor \]

This choice of $w$ determines $A_{n+1}^3 $ in lemma \ref{lemme3}. Let $A_{n+1}=A_{n+1}^3 A_{n+1}^2 A_{n+1}^1$ and $E_{n+1}= E_{n+1}^3 \cap A_{n+1}^3 (E_{n+1}^2) \cap  A_{n+1}^3 A_{n+1}^2 (E_{n+1}^1)$. 
By lemma \ref{lemme3}, we have:

\[  \| A_{n+1}^3 \|_{n+1} \leq (q_{n+1})^{R_{\ref{estdeanplusuntrois}}(n)}  \left(  \frac{2^{n+2} q_n^2 (b_nq_{n-1})^{R_{\ref{rdebnn}}(n-1)} }{q_{n+1}}  \right)^{R_{\ref{estdeanplusuntrois}}(n)}  \leq (b_{n+1}q_n)^{\rrr\label{rbis}(n)}  \]

for a fixed sequence $R_{7}(n)$. This ensures the existence of $R_{\ref{rpol3}}(n)$ such that:

\[ \|A_{n+1} \|_{n+1}  \leq  \left(b_{n+1} q_n \right)^{R_{\ref{rpol3}}(n)} \] %\|\phi_n \|_{n+1}  \] 

Moreover, we have:

\[  w \leq  \frac{q_{n+1}}{2^{n+1} q_n^2 \|B_n\|_1    }  \]

Since $q_n$ divides $q_{n+1}$, and by the left-hand side of assumption $\ref{65rotisom}$ of lemma \ref{lemmefondarotisom}, we have:

\[ \left| \frac{p_{n+1}}{q_{n+1}} - \frac{p_n}{q_n} \right| \geq \frac{1}{q_{n+1}} \]

Therefore, by the right-hand side of assumption $\ref{65rotisom}$, $q_{n+1} \geq (b_{n+1} q_n)^{R_{\ref{r0csterotisom}}(n)}$.

We will choose an explicit sequence $R_{\ref{r0csterotisom}}(n)$ such that: \[(b_{n+1} q_n)^{R_{\ref{r0csterotisom}}(n)} \geq \left( 2^{n+1} q_n^{3/2} (b_n q_{n-1})^{R_{\ref{rdebnn}}(n-1)} \right)^{2} \] This choice implies: \[ q_{n+1} \geq \left( 2^{n+1} q_n^{3/2} (b_n q_{n-1})^{R_{\ref{rdebnn}}(n-1)} \right)^{2} \]

Therefore,

\[ w \geq \frac{q_{n+1}}{2^{n+1} q_n^2 (b_n q_{n-1})^{R_{\ref{rdebnn}}(n-1)}  } -1 \geq \frac{1}{2} \frac{\left( 2^{n+1} q_n^{3/2} (b_n q_{n-1})^{R_{\ref{rdebnn}}(n-1)} \right)^{2}}{2^{n+1} q_n^2 (b_n q_{n-1})^{R_{\ref{rdebnn}}(n-1)}  }  \geq 2^n q_n \|B_n\|_1  \]

Therefore, by lemma \ref{lemme3}, for any $i=0,...,q_{n+1}-1$, we have: 

\[  diam \left( (A_{n+1}^3)^{-1} \left(   \left[0,1 \right] \times \left[ i/q_{n+1} ,  (i+1)/q_{n+1} \right[ \right) \cap E_{n+1}^{3} \right) \leq  \frac{1}{ 2^n q_n \|B_n\|_1} \]

Therefore, by lemma \ref{lemme2rotisom},

\[  diam \left(  (A_{n+1}^2)^{-1} \left(  (A_{n+1}^3)^{-1} \left( \left( \left[0,1 \right] \times \left[ i/q_{n+1} ,  (i+1)/q_{n+1} \right[  \right) \cap E_{n+1}^{3} \right) \right) \cap E_{n+1}^{2}    \right)    \leq  \frac{1}{ 2^n \|B_n\|_1}   \]

and therefore, we also have:

\[  diam \left( (A_{n+1}^1)^{-1} \left(  (A_{n+1}^2)^{-1} \left(  (A_{n+1}^3)^{-1} \left(\left(\left[0,1 \right]  \times \left[ i/q_{n+1} ,  (i+1)/q_{n+1} \right[  \right) \cap E_{n+1}^{3} \right) \right) \cap E_{n+1}^{2}    \right) \cap E_{n+1}^{1}    \right) \leq \frac{1}{2^n \|B_n\|_1}  \]

%Hence the corollary.
\end{proof}

\subsection{Construction in higher dimensions}

The construction in higher dimensions is slightly different of the previous chapter. The first two steps are the same as in dimension $2$ (we make the construction in the plan $(x_1,x_d)$, see the previous chapter), and for the third step (generation), we combine all $d-1$ dimensions. The following lemma generalizes lemma \ref{lemme3}:

\begin{lemma}
\label{lemme3gal}
For any integers $w_1,...,w_{d-1} \geq 1$ such that $2 \prod_{i=1}^{d-1} w_i \leq q_{n+1}$, there is a smooth, measure-preserving, and $S_{\frac{1}{ q_n}}$-equivariant diffeomorphism $A_{n+1}^3$, and an explicit sequence of integers $\rrr(n)\label{estdeanplusuntroisgal}$, such that:

\[ \|A_{n+1}^3 \|_{n+1}  \leq  \left(\frac{q_{n+1}}{w_1}\right)^{R_{\ref{estdeanplusuntroisgal}}(n)} \]%\|\phi_n \|_j  \] 

and there exists a $S_{\frac{1}{ q_n}}$-invariant and $A_{n+1}^3$-invariant set $E_{n+1}^3$ such that 

$\mu( E_{n+1}^{3 c}) \leq 4/2^n$, and such that for any $i=0,...,q_{n+1}-1$, we have: 

\[  diam \left( (A_{n+1}^3)^{-1} \left(  \left[0,1 \right]^{d-1} \times \left[ i/q_{n+1} ,  (i+1)/q_{n+1} \right[  \right) \cap E_{n+1}^{3} \right) \leq \max \left(  \frac{1}{w_1},...,\frac{1}{w_{d-1}}, \frac{2^{d-1}w_1...w_{d-1}}{q_{n+1}} \right)  \]

\end{lemma}

%Figure \ref{step3multidim} illustrates the action of $A_{n+1}^3$.

As in dimension 2, we let $A_{n+1}=A_{n+1}^3 A_{n+1}^2 A_{n+1}^1$ and \[ E_{n+1}= E_{n+1}^3 \cap A_{n+1}^3 (E_{n+1}^2) \cap  A_{n+1}^3 A_{n+1}^2 (E_{n+1}^1) \]

We obtain the corollary:

\begin{corollary}
\label{corogenerationgal}
There exists an explicit sequence of integers $\rrr (n) \label{rpol3gal}$ depending only on $n$, there is a smooth, measure-preserving, and $S_{\frac{1}{ q_n}}$-equivariant diffeomorphism $A_{n+1}$, such that:

\[ \|A_{n+1} \|_{n+1}  \leq  \left(b_{n+1} q_n \right)^{R_{\ref{rpol3gal}}(n)} \] %\|\phi_n \|_{n+1}  \] 

and there exists a $S_{\frac{1}{ q_n}}$-invariant and $A_{n+1}$-invariant set $E_{n+1}$ such that 

$\mu( E_{n+1}^{ c}) \leq 4/2^n$, and such that for any $i=0,...,q_{n+1}-1$, we have: 

\[  diam \left( (A_{n+1})^{-1} \left( \left[0,1 \right]^{d-1} \times \left[ i/q_{n+1} ,  (i+1)/q_{n+1} \right[ \right) \cap E_{n+1} \right) \leq \frac{1}{2^n \|B_n\|_1}   \]

\end{corollary}

%\begin{figure}[h]
%\label{anplusuntroisfig}
%\centering
%\includegraphics[height=8cm]{rotisomgener}
%\caption{Illustration of the third step (generation): a connected component of $R^{(n)}$ with $h_n(l)=3$.}
%\end{figure}

\begin{proof}[Proof of lemma \ref{lemme3gal}.]

We denote:

\[ \tilde{A}_{n+1,w}^{3}: \begin{array}[t]{lcl}  [0,1]\times \varmathbb{T}^1 &\rightarrow &  [0,1]\times \varmathbb{T}^1  \\
                               (x,y) & \mapsto    & (\tilde{A}_{n+1,1,w}^{3}(x,y),\tilde{A}_{n+1,2,w}^{3}(x,y))
           \end{array}
           \]

the map $A_{n+1}^{3}$ of the 2-dimensional case, given by lemma \ref{lemme3}, associated with the integer $w$. For $i=1,...,d-1$, we denote:

\[ A_{n+1,w}^{3,i}(x_1,...,x_d)= (x_1,...,x_{i-1}, \tilde{A}_{n+1,1,w}^{3} (x_i,x_d), x_{i+2},...,\tilde{A}_{n+1,2,w}^{3} (x_i,x_d)) \]

We let: $A_{n+1}^{3}=  A_{n+1,w_1}^{3,1} A_{n+1,w_1w_2}^{3,2}...A_{n+1,w_1...w_{d-1}}^{3,d-1}$ (see figures \ref{step3galbeforefigrotisom}, \ref{step3galafterfigrotisom}, \ref{step3galafterpostfigrotisom}). We define $E_{n+1}^3$ by analogy with lemma \ref{lemme3}.

\end{proof}

%We now prove :

\begin{proof}[Proof of corollary \ref{corogenerationgal}.] 

The proof is analogous to the proof of corollary \ref{corogenerationrotisom}. We let:

\[ w_1 = \left\lfloor  \frac{q_{n+1}}{\left( 2^{n+1} q_n (b_n q_{n-1})^{R_{\ref{rdebnn}}(n-1)} \right)^d }  \right\rfloor \]

and for $i=2,...,d-1$, $w_i=  2^{n} q_n (b_n q_{n-1})^{R_{\ref{rdebnn}}(n-1)}$. 

As in lemma \ref{lemme3}, there exists $R_{\ref{rpol3gal}}(n)$ such that:

\[ \|A_{n+1} \|_{n+1}  \leq  \left(b_{n+1} q_n \right)^{R_{\ref{rpol3gal}}(n)} \] %\|\phi_n \|_{n+1}  \] 

For $i=2,...,d-1$, we have:

\[ \frac{1}{w_i} \leq \frac{1}{2^{n} q_n (b_n q_{n-1})^{R_{\ref{rdebnn}}(n-1)}} \leq \frac{1}{2^n q_n \| B_n\|_1 } \]

Moreover, we have:

\[  \frac{2^{d-1} w_1w_2...w_{d-1}}{q_{n+1}} \leq \frac{1}{2^{n+1} q_n \|B_n\|_1}  \]

We will choose an explicit sequence $R_{\ref{r0csterotisom}}(n)$ such that: \[(b_{n+1} q_n)^{R_{\ref{r0csterotisom}}(n)} \geq \left( 2^{n} q_n (b_n q_{n-1})^{R_{\ref{rdebnn}}(n-1)} \right)^{d+1} \] This choice implies: \[ q_{n+1} \geq \left( 2^{n} q_n (b_n q_{n-1})^{R_{\ref{rdebnn}}(n-1)} \right)^{d+1} \]

Therefore,

\[  \frac{1}{w_1} \leq  \frac{\left( 2^{n} q_n (b_n q_{n-1})^{R_{\ref{rdebnn}}(n-1)} \right)^{d-1} }{q_{n+1}} \leq \frac{1}{2^n q_n \|B_n\|_1 } \]

%\[ w \geq \frac{q_{n+1}}{2^{n+1} q_n^2 (b_n q_{n-1})^{R_{\ref{rdebnn}}(n-1)}  } -1 \geq \frac{1}{2} \frac{\left( 2^{n+1} q_n^{3/2} (b_n q_{n-1})^{R_{\ref{rdebnn}}(n-1)} \right)^{2}}{2^{n+1} q_n^2 (b_n q_{n-1})^{R_{\ref{rdebnn}}(n-1)}  }  \geq 2^n q_n \|B_n\|_1  \]

By combining lemma \ref{lemme3gal} and lemma \ref{lemme2rotisom}, we obtain the corollary.

%\[  diam \left(  (A_{n+1}^2)^{-1} \left(  (A_{n+1}^3)^{-1} \left( \left( \left[0,1 \right] \times \left[ i/q_{n+1} ,  (i+1)/q_{n+1} \right[  \right) \cap E_{n+1}^{3} \right) \right) \cap E_{n+1}^{2}    \right)    \leq  \frac{1}{ 2^n \|B_n\|_1}   \]

%and therefore, we also have:

%\[  diam \left( (A_{n+1}^1)^{-1} \left(  (A_{n+1}^2)^{-1} \left(  (A_{n+1}^3)^{-1} \left(\left(\left[0,1 \right]  \times \left[ i/q_{n+1} ,  (i+1)/q_{n+1} \right[  \right) \cap E_{n+1}^{3} \right) \right) \cap E_{n+1}^{2}    \right) \cap E_{n+1}^{1}    \right) \leq \frac{1}{2^n \|B_n\|_1}  \]

%Hence the corollary.
\end{proof}

%:for $i=0,...,d-1$, $h_n \geq 1$, let $\phi_{n,q_n}^i(x_1,...,x_d)= (x_1,..,x_{i-1}, \phi_{n,h_n k_n q_n}(x_i,x_{i+1}),x_{i+2},...,x_d)$, extended by $1/h_n k_n q_n$-periodicity along the $x_d$ coordinate. We let \[ A_{n+1}^3(x_1,...,x_d)= \phi_{n,q_n}^1...\phi_{n,q_n}^{d-1}(x_1,...,x_d) \]

%For $j=1,...,d-1$, $r=0,...,h_n k_n q_n-1$, let also  \[ E_{j,r} = \left[0,1 \right]^{j-1} \times \left[ \frac{1}{2^n}, 1- \frac{1}{2^n}  \right] \times \left[ \frac{r}{h_n k_n q_n} + \frac{1}{h_n k_n q_n 2^n}, \frac{r+1}{h_n k_n q_n} - \frac{1}{h_n k_n q_n 2^n} \right] \times \left[0,1 \right]^{d-(j+1)}  \]

%(in the notations, we omit dependencies in $n$) and \[ E_j = \bigcup_{r=0}^{h_n k_n q_n-1} E_{j,r} \]

%Note that $E_j$ is $\phi_{n,q_n}^j$-invariant. We let 

%\[ E_{n+1}^3 = E- \phi_{n,q_n}^1(E_2) \bigcap...\bigcap \phi_{n,q_n}^1...\phi_{n,q_n}^{d-2}(E_{d-1}) \]

%By iterating the reasoning done in dimension $2$, we obtain, for 

%$l=0,...,(h_n k_n q_n )^d-1$:

%\[  diam \left( (A_{n+1}^3)^{-1} \Delta_{l,(h_n k_n q_n )^d} \cap E_{n+1}^{3} \right) \leq \frac{1}{h_n k_n q_n}   \]

%Finally, if we let $q_{n+1} = q'_{n+1} (h_n k_n q_n)^d$, and with a suitable choice of choice of $h_n$, we obtain the wanted estimation:

%\[  diam \left( (A_{n+1})^{-1} \left( \Delta_{l,q_{n+1}} \cap E_{n+1} \right) \right) \leq \frac{1}{2^n \|B_n\|_1}   \]

%Thus, we get generation, as in dimension $2$.

\begin{figure}[!h]
\centering
\includegraphics[height=5cm]{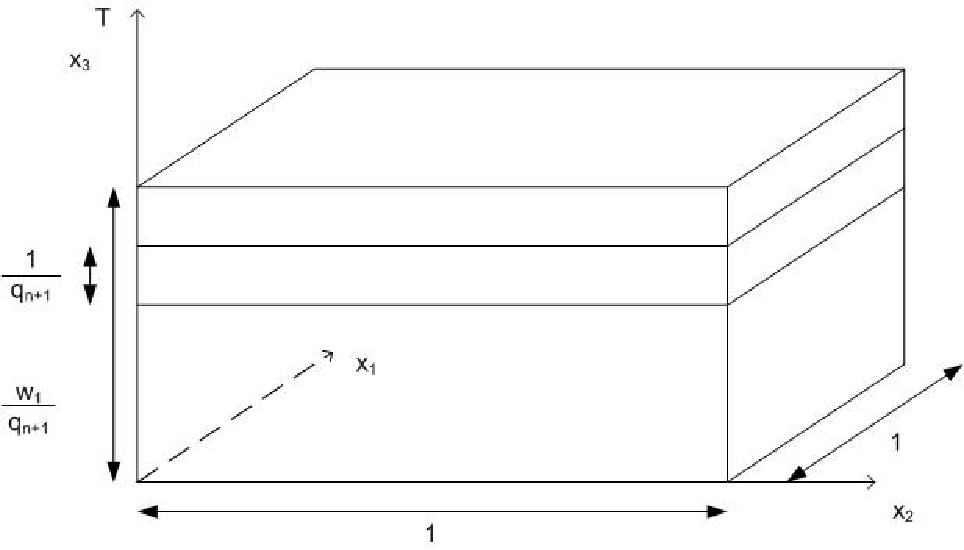}
\caption{An element $[0,1]^2 \times [i/q_{n+1}, (i+1)/q_{n+1}[$ (we take $d=3$), before the application of $(A_{n+1}^3)^{-1}= \left( A_{n+1,w_1w_2}^{3,2} \right)^{-1} \left( A_{n+1,w_1}^{3,1}\right)^{-1}$. Its size is $1 \times 1 \times 1/ q_{n+1}$.}
\label{step3galbeforefigrotisom}
\end{figure}

\begin{figure}[!h]
\centering
\includegraphics[height=5cm]{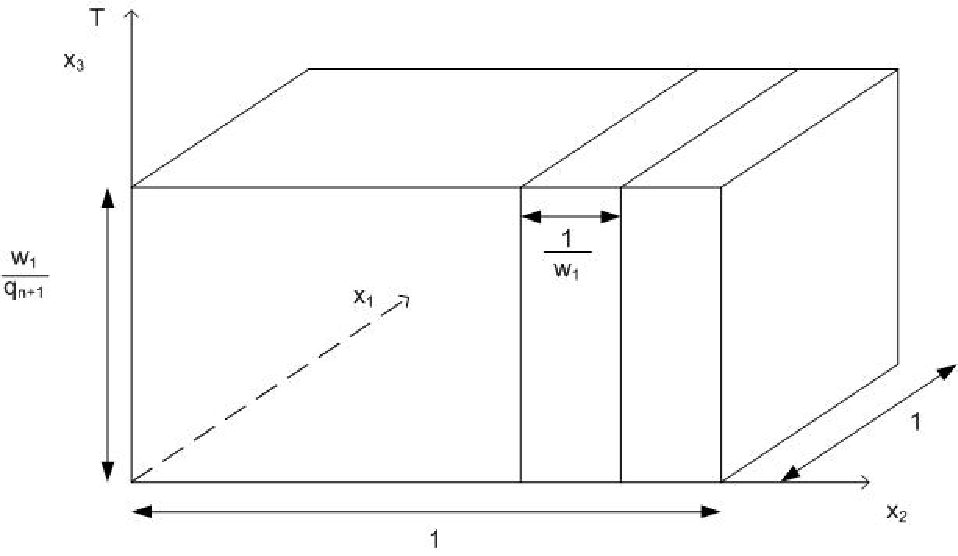}
\caption{The element $\left( A_{n+1,w_1}^{3,1}\right)^{-1} \left( [0,1]^2 \times [i/q_{n+1}, (i+1)/q_{n+1}[ \right) \bigcap E_{n+1}^3$. Its size is less than $1 \times  1/ w_1 \times w_1/ q_{n+1}$. }
\label{step3galafterfigrotisom}
\end{figure}

\begin{figure}[!h]
\centering
\includegraphics[height=5cm]{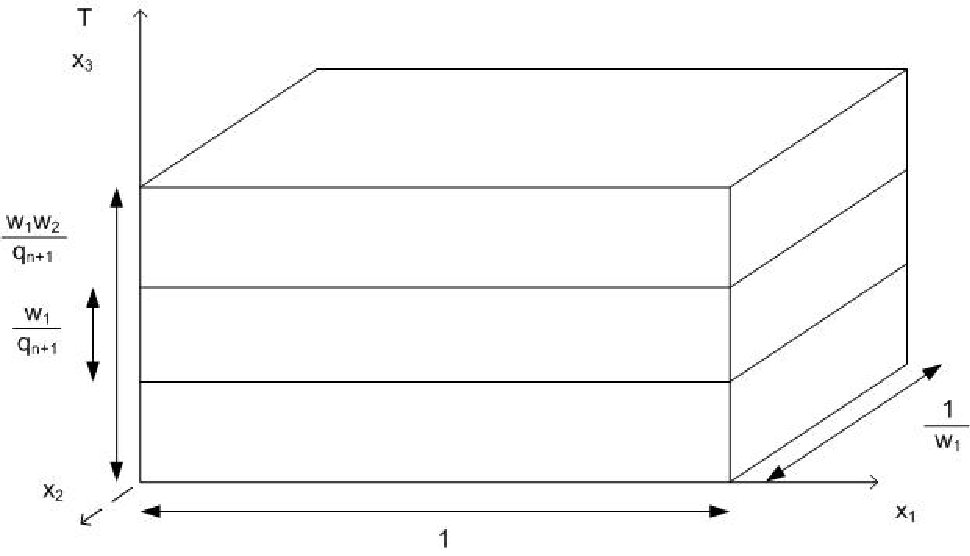}
\caption{$\left( A_{n+1,w_1}^{3,1}\right)^{-1} \left( [0,1]^2 \times [i/q_{n+1}, (i+1)/q_{n+1}[ \right) \bigcap E_{n+1}^3$, in the plan $(x_1,x_3)$.}
\label{step3galafterpostfigrotisom}
\end{figure}

%\clearpage

%After application of $\left( A_{n+1,w_1w_2}^{3,2} \right)^{-1} $, its size will be less than $1/ w_2 \times  1/ w_1 \times w_1w_2/ q_{n+1}$.
\begin{figure}[!h]
\centering
\includegraphics[height=5cm]{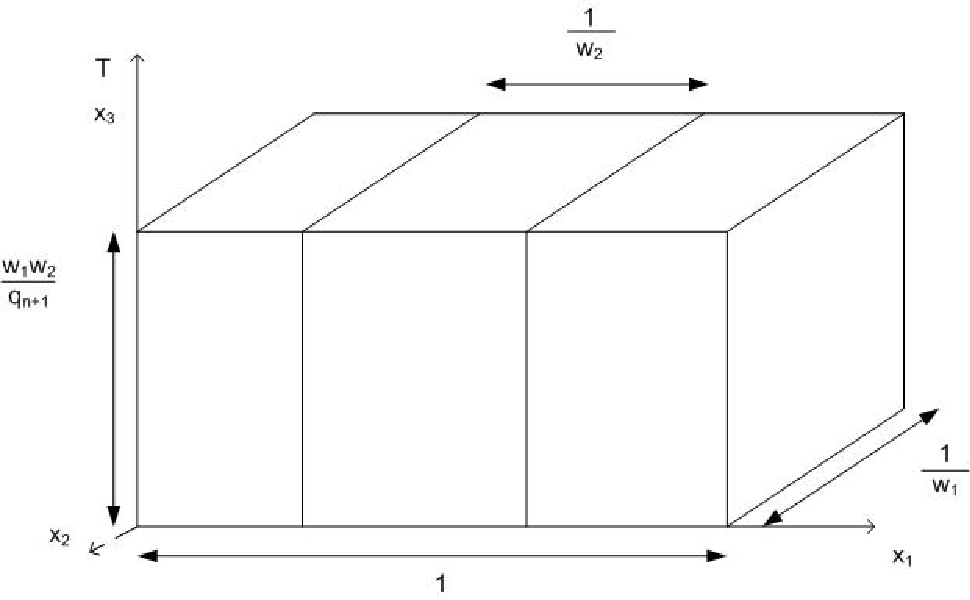}
\caption{$\left( A_{n+1,w_1w_2}^{3,2} \right)^{-1} \left( A_{n+1,w_1}^{3,1}\right)^{-1} \left( [0,1]^2 \times [i/q_{n+1}, (i+1)/q_{n+1}[ \right) \bigcap E_{n+1}^3$, in the plan $(x_1,x_3)$. Its size is less than $1/ w_2 \times  1/ w_1 \times w_1w_2/ q_{n+1}$.}
\label{step3galafterpostlaststepfigrotisom}
\end{figure}

\subsection{Convergence of the sequence of diffeomorphisms and ergodicity of the limit $T$. Proof that $T$ is a pseudo-rotation in dimension 2}

\label{cgcetn}

By combining lemma \ref{conditionsbnrotisom}, corollary \ref{corolisomrotisom}, and proposition \ref{propexistencerotisom}, and since  $\xi_n$ generates, then in order to complete the proof of lemma \ref{lemmefondarotisom}, it remains to show that $T_n= B_n^{-1} S_{\frac{p_n}{q_n}} B_n$ converges in the smooth topology, and that the limit $T$ of $T_n$ is ergodic.

To show the convergence of $T_n= B_n^{-1} S_{\frac{p_n}{q_n}} B_n $, by the Cauchy criterion, it suffices to show that $\sum_{n \geq 0} d_n(T_{n+1},T_n)$ converges. We combine the estimation of $B_{n+1}$ and the assumption \ref{65rotisom} of lemma \ref{lemmefondarotisom} of closeness between $p_{n+1}/q_{n+1}$ and $p_n/q_n$. We recall the lemma \cite[p.1812]{windsor07}:

\begin{lemma}
\label{faadirotrotisom}
Let $k \in \varmathbb{N}$. There is a constant $C(k,d)$ such that, for any $ h \in$ Diff$(M)$, $\alpha_1,\alpha_2 \in \varmathbb{R}$, we have:

\[ d_k(hS_{\alpha_1} h^{-1},hS_{\alpha_2} h^{-1} ) \leq C(k,d) \|h\|_{k+1}^{k+1} |\alpha_1-\alpha_2|  \]

\end{lemma}

%By also recalling thatand since, for $n \geq 2$,  $\|\phi_n \|_{n+1} \leq q_n^{\rrr(n)\label{rmajphin}}$ for a sequence $R_{\ref{rmajphin}}(n)$ independent of $q_n$ (because $q_n \geq 2$ for $n \geq 2$), 
Since $T_n= B_n^{-1} S_{\frac{p_n}{q_n}} B_n = B_{n+1}^{-1} S_{\frac{p_n}{q_n}} B_{n+1} $, we obtain, for a fixed sequence $\rrr(n) \label{r5csterotisom}$ (that depends on $n$ and on the dimension $d$):

\begin{eqnarray*}
d_n(T_{n+1},T_n) = d_n ( B_{n+1}^{-1} S_{\frac{p_{n+1}}{q_{n+1}}} B_{n+1}, B_{n+1}^{-1} S_{\frac{p_n}{q_n}} B_{n+1}) \leq C(k,d) \|B_{n+1}\|_{n+1}^{n+1} \left| \frac{p_{n+1}}{q_{n+1}} - \frac{p_n}{q_n} \right|  \\
\leq  \left( b_{n+1}  q_n \right)^{R_{\ref{r5csterotisom}}(n)} \left| \frac{p_{n+1}}{q_{n+1}} - \frac{p_n}{q_n} \right|  
\end{eqnarray*}

For some choice of the sequence $R_{\ref{r0csterotisom}}(n)$ in lemma \ref{lemmefondarotisom}, this last estimate guarantees the convergence of $T_n$ in the smooth topology. Moreover, the limit $T$ is ergodic, because it is metrically isomorphic to an irrational rotation of the circle, which is ergodic.

\bigskip

To show corollary \ref{corpseudorotrotisom}, let us show that $T$ is a pseudo-rotation when $d=2$.

\begin{proposition}
When $d=2$, the limit $T$ of $T_n$ is a pseudo-rotation of angle $\alpha$.
\end{proposition}

\begin{proof}

Since $T_{| \ddd M}=S_{\alpha | \ddd M}$, then $T$ is isotopic to the identity, and $(0,\alpha) \in \mbox{rot}(T)$, where $\mbox{rot}(T)$ is the set of rotation vectors of $T$.

To show the proposition, it suffices to show that in any point of $M$, the rotation vector exists and is independent of the choice of the point. Let $\tilde{T}:[0,1] \times \varmathbb{R} \rightarrow [0,1] \times \varmathbb{R}$ a lift of $T$, $\epsilon>0$ and $x \in [0,1] \times \varmathbb{R}$ and $y \in \ddd ([0,1] \times \varmathbb{R})$. For any integer $n>0$, we have:

\[ \frac{\tilde{T}^n(x)-x}{n} - \frac{\tilde{T}^n(y)-y}{n} = \frac{\tilde{T}^n(x)-\tilde{T}^n(y)}{n} \]

Let $m$ such that: 

\[ \sum_{p \geq m} \|B_{p+1}\|_{1} \left| \frac{p_{p+1}}{q_{p+1}} - \frac{p_p}{q_p} \right|  \leq \epsilon \]

We have: \[ d_1(\tilde{T}^n,\tilde{T}_m^n) \leq \sum_{p \geq m} d_1(T_{p+1}^n,T_p^n) = d_1 ( B_{p+1}^{-1} S_{\frac{np_{p+1}}{q_{p+1}}} B_{p+1}, B_{p+1}^{-1} S_{\frac{np_p}{q_p}} B_{p+1}) \leq n \epsilon \]

Moreover, \[ \tilde{T}^n(x)-\tilde{T}^n(y) =  \tilde{T}^n(x)-\tilde{T}_m^n(x) + \tilde{T}_m^n(x)-\tilde{T}_m^n(y) + \tilde{T}_m^n(y)-\tilde{T}^n(y) \]

\[ \leq 2 d_1(\tilde{T}^n,\tilde{T}_m^n) + |\tilde{T}_m^n(x)-\tilde{T}_m^n(y)| \leq 2n \epsilon + | \tilde{T}_m^n(x)-\tilde{T}_m^n(y)| \] 

Moreover, \[ |\tilde{T}_m^n(x)-\tilde{T}_m^n(y)| \leq \| \tilde{T}_m^n \|_1 |x-y| \leq \| B_m \|_1^2 |x-y|    \]

Therefore, for any $n$ sufficiently large,

\[ \left|\frac{\tilde{T}^n(x)-\tilde{T}^n(y)}{n}\right| \leq 2 \epsilon + \frac{1}{n}  \| B_m \|_1^2 |x-y| \leq 3 \epsilon \]

Therefore, the translation vector of $x$ exists and is equal to the translation vector of $y$, which is $(0,\tilde{\alpha})$, where $\tilde{\alpha}$ is a lift of $\alpha$.

We conclude that $\mbox{rot}(T)= \{ (0,\alpha) \}$.

\end{proof}

%By taking $m \rightarrow + \infty $ in relation (\ref{relationtnm}), we obtain that 

\subsection{Extension to more general manifolds}

To extend the construction from $[0,1]^{d-1} \times \varmathbb{T}$ to a general $d$-dimensional smooth compact connected manifold $M$, admitting an effective volume-preserving circle action $\hat{S}_t$, we proceed as in \cite[p. 1805]{windsor07} and the previous chapter. We keep denoting $S_t$ the circle action on $[0,1]^{d-1} \times \varmathbb{T}$. For $q \geq 1$, let $F_q$ be the set of fixed points of $\hat{S}_{1/q}$. Let $B= \ddd M \bigcup_{q \geq 1} F_q$ be the set of exceptional points. We recall the proposition:

\begin{proposition}[\cite{windsor07}]
Let $M$ be a $d$-dimensional smooth compact connected manifold, with an effective  circle action $\hat{S}_t$, preserving a smooth volume $\mu$. Let $S_t$ denote the circle action on $[0,1]^{d-1} \times \varmathbb{T}$. There exists a continuous surjective map $\Gamma: [0,1]^{d-1} \times \varmathbb{T} \rightarrow M$ such that:

\begin{enumerate}
\item the restriction of $\Gamma$ to $]0,1[^{d-1} \times \varmathbb{T}$ is a smooth diffeomorphic embedding.
\item $\mu(\Gamma(\ddd([0,1]^{d-1} \times \varmathbb{T}) )) =0$
\item $B \subset \Gamma(\ddd([0,1]^{d-1} \times \varmathbb{T}) )$
\item $\Gamma_*(Leb)=\mu$
\item $\hat{S}_t\Gamma=\Gamma S_t$
\end{enumerate}

\end{proposition}

We use this proposition at each step to apply lemma \ref{lemmekatokrotisom}. We let $\hat{T}_n: M \rightarrow M$ defined by $\hat{T}_n(x) = \Gamma B_n^{-1} S_{\frac{p_n}{q_n}} B_n \Gamma^{-1} (x)$ if $x \in \Gamma (]0,1[^{d-1} \times \varmathbb{T})$ and $\hat{T}_n(x)= \hat{S}_{\frac{p_n}{q_n}}(x)$ otherwise. To show that $\hat{T}_n$ is a smooth diffeomorphism (which implies that its limit is also smooth), we use the facts that $\Gamma_{|]0,1[^{d-1} \times \varmathbb{T}}$ is a smooth diffeomorphism, than $B_n=Id$ on a neighborhood of $\ddd([0,1]^{d-1} \times \varmathbb{T}$ and that $\hat{S}\Gamma=\Gamma S$. To construct the metric isomorphism $\hat{K}_n^\infty= \Gamma \bar{K}_n^\infty$, we use the fact that the restriction of $\Gamma$ to a set of full measure is a metric isomorphism. Details are in the previous chapter.

Finally, to show that $\hat{T} \in \mathcal{A}_\alpha$, where $\hat{T}$ is the limit of $\hat{T}_n$ in the smooth topology, we let $\hat{H}_n: M \rightarrow M$ defined by $\hat{H}_n(x)=  \Gamma B_n \Gamma^{-1} (x) $ if $x \in \Gamma (]0,1[^{d-1} \times \varmathbb{T})$ and $\hat{H}_n(x)=x$ otherwise. We write \[ \hat{T}- \hat{H}_n^{-1} \hat{S}_{\alpha} \hat{H}_n= \hat{T} - \hat{T}_n+\hat{T}_n- \hat{H}_n^{-1} \hat{S}_{\alpha} \hat{H}_n \]

We know that $\hat{T} - \hat{T}_n \rightarrow 0$ in the smooth topology. We show that $\hat{T}_n- \hat{H}_n^{-1} \hat{S}_{\alpha} \hat{H}_n \rightarrow 0$ in the smooth topology by proceeding as in the proof of the convergence of $T_n$ on $[0,1]^{d-1} \times \varmathbb{T}$ in subsection \ref{cgcetn}.

%\section{Conclusion and perspectives}
%\label{cclpers}

%In the first step of the construction of the diffeomorphism, we proceed slightly differently than in Katok's original construction: he rather takes . This difference is small but critical: following his original method, we find ourselves having to quasi-permute $h_n k_n q_n$ slices to matching the location of $R^{(n)}$. We were able to do it in a smooth measure-preserving way, but only with more than $q_n$ iterations, thus jeopardizing the polynomial estimation in $q_n$, and ultimately the obtention of all Liouville numbers. This suggest the following questions:

%\textit{Question 1}: 

%
% with the following estimation:.

\bibliography{refdyn}

\end{document}